\documentclass[a4paper,11pt,english]{smfart}
\usepackage{amsfonts}
\usepackage{amsmath,amsthm}
\usepackage{latexsym}
\usepackage{array}
\usepackage{amssymb}
\usepackage{graphicx}
\usepackage{enumerate}
\usepackage{verbatim}
\usepackage{float}
\usepackage[english]{babel}
\usepackage{color}
\definecolor{marin}{rgb}   {0.,   0.3,   0.7} 
\definecolor{rouge}{rgb}   {0.8,   0.,   0.} 
\definecolor{sepia}{rgb}   {0.8,   0.5,   0.} 
\usepackage[colorlinks,citecolor=marin,linkcolor=rouge,
            bookmarksopen,
            bookmarksnumbered
           ]{hyperref}

\newtheorem{thm}{Theorem}{\bf}{\it}
\newtheorem{lem}[thm]{Lemma}{\bf}{\it}
{\bf}{\it}
{\bf}{\it}
{\bf}{\it}
{\bf}{\it}
{\bf}{\it}
{\bf}{\it}
{\bf}{\it}
{\bf}{\it}
{\bf}{\it}

\theoremstyle{definition}
\newtheorem{defn}[thm]{Definition}{\bf}{\rm}
{\bf}{\rm}
\newtheorem{ass}[thm]{Assumption}{\bf}{\rm}
\newtheorem{rem}[thm]{Remark}{\bf}{\rm}

\newcommand{\nab}{\langle \nabla \rangle_c}

\newcommand {\aplt} {\ {\raise-.5ex\hbox{$\buildrel<\over\sim$}}\ } 

\newcommand{\Z}{\mathbb{Z}}
\newcommand{\e}{\mathrm{e}}
\newcommand{\ua}{u_\ast}
\newcommand{\va}{v_\ast}
\newcommand{\Ac}{\mathcal{A}_c}

\newcommand{\U}{\mathcal{U}_{\ast}(t_n)}

\newcommand{\Un}{\mathcal{U}_{\ast}^n}

\newcommand{\dd}{\mathrm{d}}

\author{Simon Baumstark}
\address{Fakult\"{a}t f\"{u}r Mathematik, Karlsruhe Institute of Technology,
Englerstr. 2, 76131 Karlsruhe, Germany}
\email{simon.baumstark@kit.edu}

\author{Erwan Faou}
\address{INRIA \& ENS Rennes  \\
Avenue Robert Schumann F-35170 Bruz, France. } 
\email{Erwan.Faou@inria.fr}

\author{Katharina Schratz}
\address{Fakult\"{a}t f\"{u}r Mathematik, Karlsruhe Institute of Technology,
Englerstr. 2, 76131 Karlsruhe, Germany}
\email{katharina.schratz@kit.edu}

\makeatletter
\def\@makefnmark{\hbox{$\m@th^{\@thefnmark}$}}
\makeatother
\title[]{Uniformly accurate exponential-type integrators for Klein-Gordon equations with asymptotic convergence to  the classical NLS splitting}

\begin{document}

\begin{abstract}
We introduce efficient and robust exponential-type integrators for Klein-Gordon equations which resolve the solution  in the relativistic regime as well as in the highly-oscillatory non-relativistic regime without any step-size restriction under the same regularity assumptions on the initial data  required for the integration of the corresponding  nonlinear Schr\"odinger  limit system. In contrast to previous works we do not employ any asymptotic/multiscale expansion of the solution. This allows us to derive uniform convergent schemes under far weaker regularity assumptions on the exact solution.   In addition, the newly derived first- and second-order exponential-type integrators converge to the classical Lie, respectively, Strang splitting in the nonlinear Schr\"odinger limit.
\end{abstract}

%\thanks{The authors 
%}

\maketitle
%\tableofcontents

\section{Introduction}
Cubic Klein-Gordon equations
\begin{equation}
\begin{aligned}\label{eq:kgr}
& c^{-2} \partial_{tt} z - \Delta z + c^2 z = \vert z\vert^{2} z, \quad z(0,x) = z_0(x),\quad \partial_t z(0,x) = c^2 z'_0(x)\\
\end{aligned}
\end{equation}
are extensively studied numerically in the relativistic regime $c=1$, see \cite{Gau15,StVaz78} and the references therein. In contrast, the so-called ``non-relativistic regime'' $c\gg 1$ is numerically much more involved due to the highly-oscillatory behavior of the solution. We refer to \cite{EFHI09,HLW} and the references therein for an introduction and overview on highly-oscillatory problems.\\

Analytically, the non-relativistic limit regime $c\to \infty$ is well understood nowadays: The exact solution $z$ of \eqref{eq:kgr} allows (for sufficiently smooth initial data) the expansion 
\[
z(t,x) = \frac{1}{2}\left( \e^{ic^2 t} u_{\ast,\infty}(t,x) + \e^{-ic^2t} \overline{v}_{\ast,\infty}(t,x) \right) +  \mathcal{O}(c^{-2})
\]
 on a time-interval uniform in $c$, where $(u_{\ast,\infty}, v_{\ast,\infty})$ satisfy the cubic Schr\"odinger limit system
\begin{equation}\label{NLSlimit}
\begin{aligned}
i \partial_t u_{\ast, \infty} &=& \frac{1}{2} \Delta u_{\ast,\infty} + \frac{1}{8}\big(\left \vert  u_{\ast,\infty}\right\vert^2 + 2\left\vert v_{\ast,\infty}\right\vert^2\big)  u_{\ast,\infty}\qquad u_{\ast,\infty}(0) = \varphi - i \gamma
\\
i \partial_t v_{\ast, \infty} &=& \frac{1}{2} \Delta v_{\ast,\infty} + \frac{1}{8}\big(\left \vert  v_{\ast,\infty}\right\vert^2 + 2\left\vert u_{\ast,\infty}\right\vert^2\big)  v_{\ast,\infty},\qquad v_{\ast,\infty}(0) = \overline{\varphi} - i \overline{\gamma}
\end{aligned}
\end{equation}
with initial values
\begin{align*}
z(0,x) \stackrel{c\to\infty}{\longrightarrow} \gamma(x) \quad \text{and}\quad c^{-1}\left(c^2-\Delta\right)^{-1/2} \partial_t z(0,x) \stackrel{c\to\infty}{\longrightarrow} \varphi(x),
\end{align*}
see \cite[Formula (1.3)]{MaNak02} and for the periodic setting \cite[Formula (37)]{FS13}.\\

Also numerically, the non-relativistic limit regime $c \gg 1$ has recently gained a lot of attention: Gautschi-type methods (see \cite{HoLu99}) are analyzed in \cite{BG}. However, due to the difficult structure of the problem they suffer from a  severe  time-step restriction as they introduce a global error of order $c^4 \tau^2$ which requires the CFL-type condition $c^2 \tau <1$. To overcome this difficulty so-called limit integrators which  reduce the highly-oscillatory problem to  the corresponding non-oscillatory limit system  (i.e., $c\to \infty$ in \eqref{eq:kgr}) as well as uniformly accurate schemes based on multiscale expansions were introduced in \cite{FS13} and \cite{BaoZ,ChC}.  In the following we give a comparison of these methods focusing on their convergence rates and regularity assumptions:

\emph{Limit  integrators:} Based on the modulated Fourier expansion of the exact solution (see \cite{CoHaLu03,HLW}) numerical schemes for the Klein-Gordon equation in the strongly non-relativistic limit regime $c \gg 1$ were introduced in \cite{FS13}. The benefit of this ansatz is that it allows us to  reduce  the highly-oscillatory problem \eqref{eq:kgr} to the integration of the corresponding \emph{non-oscillatory limit Schr\"odinger equation} \eqref{NLSlimit}. The latter can be carried out very efficiently without imposing any $c-$dependent step-size restriction.  However, as this approach is based on the asymptotic expansion of the solution with respect to $c^{-2}$, it only allows error bounds of order
$$\mathcal{O}(c^{-2} + \tau^2)$$
when integrating the limit system with a second-order method.  Henceforth, the limit integration method  only yields an accurate approximation of the exact solution for sufficiently large values of $c$.

\emph{Uniformly accurate schemes based on multiscale expansions:} Uniformly accurate schemes, i.e., schemes that work well for small as well as for large values of $c$ were recently introduced for Klein-Gordon equations in \cite{BaoZ,ChC}.  The idea is thereby based on a multiscale expansion of the exact solution.   We also refer to  \cite{BDZ14} for the construction and analysis in the case of highly-oscillatory second-order ordinary differential equations.  The multiscale time integrator (MTI) pseudospectral method derived in  \cite{BaoZ} allows two independent error bounds at order
$$
\mathcal{O}(\tau^2 + c^{-2})\quad \text{and} \quad \mathcal{O}(\tau^2 c^2)
$$
for sufficiently smooth solutions. These error bounds immediately imply that the MTI method converges uniformly in time with linear convergence rate at $\mathcal{O}(\tau)$ for all $c \geq 1$ thanks to the observation that
$
\mathrm{min}(c^{-2}, \tau^2c^2) \leq \tau
$. However, the optimal quadratic
convergence rate at $\mathcal{O}(\tau^2)$ is only achieved in the regimes when either $0 < c = \mathcal{O}(1)$ (i.e., the relativistic regime) or $ \frac{1}{\tau} \leq c $ (i.e., the strongly non-relativistic regime).   In the context of ordinary differential equations similar error estimates were established for MTI methods   in \cite{BDZ14}.   The first-order uniform convergence of the MTI-FP method \cite{BaoZ} holds for sufficiently smooth solutions:  First-order convergence in time holds in $H^2$ uniformly in $c$ for solutions in $H^7$ with $\sup_{0\leq t \leq T} \Vert z(t)\Vert_{H^{7}} + c^{-2}\Vert \partial_t z(t) \Vert_{H^6} \leq 1$ (see \cite[Theorem 4.1]{BaoZ}). First-order uniform convergence also  holds in $H^1$ under weaker regularity assumptions, namely for solutions in $H^6$ satisfying $
\sup_{0\leq t \leq T} \Vert z(t)\Vert_{H^{6}} + c^{-2}\Vert \partial_t z(t) \Vert_{H^5} \leq 1$ if an additional CFL-type condition is imposed in space dimensions $d=2,3$ (see \cite[Theorem 4.9]{BaoZ}).

A second-order uniformly accurate scheme based on the \emph{Chapman-Enskog expansion} was derived in \cite{ChC} for the Klein-Gordon equation. Thereby, to control the remainders in the expansion, second-order uniform convergence in $H^r$ ($r>d/2$)  requires sufficiently smooth solutions with in particular  $z(0) \in H^{r+10}$. Also,  due to the expansion, the \emph{problem needs to be considered  in $d+1$ dimensions}.\\

 We establish exponential-type integrators which converge with \emph{second-order accuracy in time uniformly in all $c>0$}. In comparison, the multiscale time integrators (MTI) derived in \cite{BaoZ,BDZ14}  only converge with first-order accuracy uniformly in all $c \geq 1$. This is due to the fact that the MTI methods are based on the multiscale decomposition
$$
z(t,x) = \e^{it c^2} z_{+}^n(t,x) + \e^{-i tc^2} \overline{z^n_{-}}(t,x) + r^n(t,x)
$$
which leads to a coupled \emph{second-order system in time} in the $c^{2}$-frequency waves $z_{\pm}^n$ and the rest frequency waves $r^n$ (cf. \cite[System (2.4)]{BaoZ}) and only allows numerical approximations at order $\mathcal{O}(\tau^2 + c^{-2})$ and $\mathcal{O}(\tau^2 c^2)$.

In contrast to  \cite{BaoZ,ChC,FS13} \emph{we do not employ any asymptotic/multiscale expansion} of the solution, but construct exponential-type integrators based on the following strategy:
\begin{itemize}
\item[1.]  In a first step we reformulate the Klein-Gordon equation \eqref{eq:kgr} as a coupled \emph{first-order system in time} via the transformations
\[
u = z - i \big(c\sqrt{-\Delta +c^2}\big)^{-1} \partial_t z, \quad v = \overline{z}-i \big(c\sqrt{-\Delta +c^2}\big)^{-1} \partial_t \overline{z}.
\]
\item[2.] In a second step we rescale the coupled first-order system in time by looking at the so-called ``twisted variables''
\[
\ua(t) = \e^{ic^2 t} u(t), \qquad \va(t) = \e^{-ic^2t}v(t).
\]
This essential step will later on allow us to treat the highly-oscillatory phases $\mathrm{e}^{\pm ic^2 t}$ and their interaction explicitly.
\item[3.] Finally, we iterate Duhamel's formula in $(\ua(t),\va(t))$ and integrate the interactions of the highly-oscillatory phases exactly by approximating only the slowly varying parts.
\end{itemize}
This strategy in particular allows us to construct uniformly accurate exponential-type integrators  up to order two which in addition asymptotically converge to the classical splitting approximation of the corresponding nonlinear Schr\"odinger limit system \eqref{NLSlimit} given in \cite{FS13}. More precisely, the second-order exponential-type integrator converges for $c \to \infty$ to the classical Strang splitting scheme
 \begin{equation}\label{limitscheme2a}
\begin{aligned}
u_{\ast,\infty}^{n+1}  &= \e^{-i \frac{\tau}{2} \frac{\Delta}{2} }  \mathrm{e}^{-i \tau \frac{3}{8}\vert\e^{-i \frac{\tau}{2} \frac{\Delta}{2} } v_{\ast, \infty}^n\vert^2} \e^{-i \frac{\tau}{2} \frac{\Delta}{2} }  u_{\ast,\infty}^n,\qquad u_{\ast,\infty}^0 = \varphi - i \gamma
\end{aligned}
\end{equation}
associated to the nonlinear Schr\"odinger limit system \eqref{NLSlimit} (see also Remark \ref{remarkStrangOk19}) where for simplicity we assumed that  $z$ is real-valued such that $\ua = \va$. A similar result holds for the asymptotic convergence of the first-order exponential-type integration scheme towards the classical Lie splitting approximation (see also Remark \ref{rem:limit1}).

In \cite{FS13} the Strang splitting \eqref{limitscheme2a} is precisely proposed for the numerical approximation of non-relativistic Klein-Gordon solutions. However, in contrast to the uniformly accurate exponential-type integrators derived here, the scheme in \cite{FS13} only yields second-order convergence in the strongly non-relativistic regime $c > \frac{1}{\tau}$ due to its error bound at order $\mathcal{O}(\tau^2+c^{-2})$.\\
 
% 
%In contrast to  \cite{BaoZ,ChC,FS13} \emph{we do not employ any asymptotic/multiscale expansion} of the solution. This allows us to derive uniform convergent schemes under weaker regularity assumptions on the exact solution.  More precisely, we establish first-, respectively, second-order convergence in time uniformly in $c$ under the same regularity assumptions as needed for first-, respectively, second-order convergence of classical integration schemes applied to the corresponding nonlinear Schr\"odinger limit system \eqref{NLSlimit}, i.e.,  $c\to \infty$ in \eqref{eq:kgr}.\\\\
 
The main novelty in this work thus lies in the development and analysis of efficient and robust exponential-type integrators for the cubic Klein-Gordon equation \eqref{eq:kgr} which 
\begin{itemize}

 \item[$\circ$] allow second-order convergence uniformly in all $c>0$ without adding an extra dimension to the problem.  

\item[$\circ$] resolve the solution $z$ in the relativistic regime $c = 1$ as well as in the non-relativistic regime $c \to \infty$ without any $c-$dependent step-size restriction  under the same regularity assumptions as needed for the integration of the corresponding limit system.

\item[$\circ$]  in addition to converging uniformly in $c$, converge asymptotically  to the classical Lie, respectively, Strang splitting for the corresponding nonlinear Schr\"odinger limit system \eqref{NLSlimit} in the non-relativistic limit $c \to \infty$.\\
\end{itemize}

Our strategy also applies to general polynomial nonlinearities $f(z) = \vert z\vert^{2p}z$ with $p \in \mathbb{N}$. However, for notational simplicity, we will focus only on the cubic case $p=1$. Furthermore, for practical implementation issues we impose periodic boundary conditions, i.e., $x \in \mathbb{T}^d$.\\

We commence in Section \ref{sec:scale}  with rescaling the Klein-Gordon equation \eqref{eq:kgr} which then allows us to construct first-, and second-order schemes that converge uniformly in $c$, see Section \ref{sec:scheme1} and \ref{sec:scheme2}, respectively.  

\section{Scaling for uniformly accurate schemes}\label{sec:scale}
In a first step we reformulate the Klein-Gordon equation \eqref{eq:kgr} as a first-order system in time which allows us to resolve the limit-behavior of the solution, i.e., its behavior for $c\to \infty$ (see also \cite{MaNak02,FS13}).

For a given $c > 0$, we define the operator
\begin{align}\label{nab}
\nab = \sqrt{- \Delta + c^2}. 
\end{align}
With this notation, equation \eqref{eq:kgr} can be written as
\begin{equation}\label{eq:kgrr}
\partial_{tt} z + c^2 \nab^2 z = c^2 f(z)
\end{equation}
with the nonlinearity
\[
 f(z) = \vert z \vert^2 z.
 \]
In order to rewrite the above equation as a first-order system in time, we set
\begin{equation}\label{eq:uv}
u = z - i c^{-1}\nab^{-1} \partial_t z , \qquad v = \overline z - i c^{-1}\nab^{-1} \partial_t \overline z
\end{equation}
such that in particular
\begin{equation}\label{eq:zuv}
z = \frac12 (u + \overline{v}).
\end{equation}
\begin{rem}\label{rem:realz}
If $z$ is real, then $u \equiv v$. 
\end{rem}
A short calculation shows that in terms of the variables $u$ and $v$ equation \eqref{eq:kgrr} reads
\begin{equation}
\label{eq:NLSc}
\begin{array}{rcl}
i \partial_t u &=& -c\nab u + c\nab^{-1} f(\textstyle \frac12 (u + \overline v)), \\[2ex]
i \partial_t v &=& -c\nab v  + c\nab^{-1} f(\textstyle \frac12 (\overline u +  v))
\end{array}
\end{equation}
with the initial conditions (see \eqref{eq:kgr})
\begin{equation}
\label{eq:BCc}
u(0) = z(0) -ic^{-1}\nab^{-1} z'(0) , \quad \mbox{and}\quad v(0) =\overline{z(0)} -ic^{-1}\nab^{-1} \overline{z'(0)}. 
\end{equation}
Formally, the definition of $\nab$ in \eqref{nab} implies that
\begin{align}\label{exnab}
c\nab \quad =\quad c^2 \quad + \quad \text{``lower order terms in $c$''}.
\end{align}
This observation motivates us to look at the so-called ``twisted variables'' by filtering out the highly-oscillatory parts explicitly: More precisely, we set
\begin{align}\label{psi}
 \ua( t) = \mathrm{e}^{-ic^2 t} u(t), \qquad \va(t) = \e^{-ic^2 t} v(t).
\end{align}
This idea of ``twisting'' the variable is well known in numerical analysis, for instance  in the context of the modulated Fourier expansion \cite{CoHaLu03,HLW},  adiabatic integrators  \cite{LJL05,HLW} as well as Lawson-type Runge--Kutta methods \cite{Law67}. In the case of ``multiple high frequencies'' it  is also widely used in the analysis of partial differential equations in low regularity spaces (see for instance \cite{Bour93})  and has been recently successfully employed numerically for the construction of low-regularity exponential-type integrators for the  KdV and Schr\"odinger equation, see \cite{HS16,OS16}.
\\

In terms of $(\ua,\va)$ system \eqref{eq:NLSc} reads  (cf. \cite[Formula (2.1)]{MaNak02})
\begin{equation}\label{eq:ua1}
\begin{aligned}
 i\partial_t \ua &=& - \Ac \ua+ c \nab^{-1} \mathrm{e}^{-ic^2t}f \left( \textstyle \frac12  ( \e^{ic^2t} \ua+ \e^{-ic^2t} \overline{\va})\right)\\
  i\partial_t \va &=& - \Ac \va+ c \nab^{-1} \mathrm{e}^{-ic^2t} f \left( \textstyle \frac12  ( \e^{ic^2t} \va+ \e^{-ic^2t} \overline{\ua})\right)
\end{aligned}
\end{equation}
with the leading operator
\begin{equation}\label{Ac}
\Ac : = c\nab - c^2.
\end{equation}
\begin{rem}\label{rem:adstar}
The advantage of looking  numerically at $(\ua,\va)$ instead of $(u,v)$  lies in the fact that the leading operator $-c\nab$ in system \eqref{eq:NLSc} is of order $c^2$ (see \eqref{exnab}) whereas its counterpart $-\Ac$ in system \eqref{eq:ua1} is ``of order one in $c$'' (see Lemma \ref{lem:boundAc} below).
\end{rem}
In the following we construct integration schemes for \eqref{eq:ua1} based on Duhamel's formula
\begin{equation}\label{du0}
\begin{aligned}
 \ua(t_n+\tau) & = \e^{i \tau \Ac} \ua(t_n)\\ & - i c \nab^{-1} \int_0^\tau \e^{i(\tau-s) \Ac}\mathrm{e}^{-ic^2(t_n+s)}f \left( \textstyle \frac12  ( \e^{ic^2(t_n+s)} \ua(t_n+s)+ \e^{-ic^2(t_n+s)} \overline{\va}(t_n+s))\right)  \mathrm{d}s,\\
 \va(t_n+\tau)&  = \e^{i \tau \Ac} \va(t_n)\\ & - i c \nab^{-1} \int_0^\tau \e^{i(\tau-s) \Ac}\mathrm{e}^{-ic^2(t_n+s)}f \left( \textstyle \frac12  ( \e^{ic^2(t_n+s)} \va(t_n+s)+ \e^{-ic^2(t_n+s)} \overline{\ua}(t_n+s))\right)  \mathrm{d}s.
\end{aligned}
\end{equation}
Thereby, to guarantee  \emph{uniform convergence  with respect to $c$} we make the following important observations. We define the Sobolev norm on $\mathbb{T}^d$ by the formula 
$$
\Vert u \Vert_r^2 = \sum_{k \in \Z^d} (1 + |k|^{2})^r | \hat u_kÊ|^2, \quad \mbox{where} \quad \hat u_k = \frac{1}{(2\pi)^d} \int_{\mathbb{T}^d} u(x) e^{i k \cdot x} d x,
$$
where for $k = (k_1,\ldots, k_d) \in \Z^d$, we set $k \cdot x = k_1 x_1 + \cdots k_d x_d$ and $|k|^2 = k\cdot k$. Moreover, for a given linear bounded operator $L$ we denote by $\Vert L \Vert_r$ its corresponding induced norm. 
\begin{lem}[Uniform bound on the operator $\Ac$] \label{lem:boundAc}\label{lem:bAc}  For all $c \in \mathbb{R}$ we have that
\begin{align}\label{boundAc}
\Vert \Ac u \Vert_r \leq \frac12\Vert u \Vert_{r+2}. 
\end{align}
\end{lem}
\begin{proof}
The operator $\Ac$ acts a the Fourier multiplier $(\Ac)_k = c^2 - c\sqrt{c^2+|k|^2}$, $k \in \Z^d$. 
Thus, the assertion follows thanks to the bound
\begin{equation*}
\begin{aligned}
\Vert \Ac u \Vert_r^2 &= \sum_{k\in \Z^d} (1 +  \vert k \vert^{2})^r \left( c\sqrt{c^2+|k|^2} - c^2 \right)^2 \vert \hat u_k\vert^2 \leq  \sum_{k\in \Z^d}(1 +  \vert k \vert^{2})^r \left( \frac{|k|^2}{2}\right)^2 \vert \hat u_k\vert^2,
\end{aligned}
\end{equation*}
where we have used that $\sqrt{1+x^2} \leq 1+ \frac{1}{2}x^2$ for all $x \in \mathbb{R}$.
\end{proof}
\begin{lem}\label{lem:expo}
For all $t \in \mathbb{R}$ we have that
\begin{equation}\label{approx1}
\begin{aligned}
\Vert \e^{i t \Ac} \Vert_r = 1 \quad \mbox{and}\quad 
\left\Vert \left (\e^{-i t \Ac} - 1\right) u\right \Vert_r \leq \frac12 \vert  t \vert \Vert u \Vert_{r+2}. 
\end{aligned}
\end{equation}
\end{lem}
\begin{proof}
The first assertion is obvious and the 
second  follows thanks to the estimate $\vert (\e^{ix}-1)\vert \leq  \vert x\vert$ which holds for all $x \in \mathbb{R}$ together with the essential bound on the operator $\Ac$ given in \eqref{boundAc}.
\end{proof}
In particular,  the time derivatives  $(\ua'(t),\va'(t))$ can be bounded uniformly in $c$.
\begin{lem}[Uniform bounds on the derivatives $(\ua'(t),\va'(t))$]\label{lem:upc}
Fix $r>d/2$. Solutions of \eqref{eq:ua1} satisfy
\begin{equation}\label{approx2}
\begin{aligned}
\Vert \ua(t_n+s) - \ua(t_n) \Vert_r & \leq \frac12 \vert s \vert \Vert \ua(t_n) \Vert_{r+2} + \frac18 \vert s \vert  \sup_{0 \leq \xi \leq s} \big( \Vert \ua(t_n+\xi)\Vert_r+ \Vert \va(t_n+\xi)\Vert_r\big)^3,\\
\Vert \va(t_n+s) - \va(t_n) \Vert_r & \leq \frac12 \vert s \vert \Vert \va(t_n) \Vert_{r+2} + \frac18 \vert s \vert  \sup_{0 \leq \xi \leq s} \big( \Vert \ua(t_n+\xi)\Vert_r+ \Vert \va(t_n+\xi)\Vert_r\big)^3.
\end{aligned}
\end{equation}
\end{lem}
\begin{proof}
The assertion follows thanks to Lemma \ref{lem:expo} together with the bound
\begin{equation}\label{cnabm}
\Vert c \nab^{-1}\Vert_r \leq 1
\end{equation}
which implies by Duhamel's perturbation formula \eqref{du0} that
\begin{equation*}
\begin{aligned}
\Vert \ua(t_n+s) - \ua(t_n) \Vert_r & \leq \vert s \vert \Vert \Ac \ua(t_n) \Vert_r + \frac18 \vert s \vert \Vert c \nab^{-1} \Vert_r \sup_{0 \leq \xi \leq s} \big( \Vert \ua(t_n+\xi)\Vert_r+ \Vert \va(t_n+\xi)\Vert_r\big)^3\\
& \leq \frac12 \vert s \vert \Vert \ua(t_n) \Vert_{r+2} + \frac18 \vert s \vert  \sup_{0 \leq \xi \leq s}  \big( \Vert \ua(t_n+\xi)\Vert_r+ \Vert \va(t_n+\xi)\Vert_r\big)^3.
\end{aligned}
\end{equation*}
Similarly we can establish the bound on the derivative $\va'(t)$.
\end{proof}

We will also employ the so-called ``$\varphi_j$ functions'' given in the following Definition.
\begin{defn}[$\varphi_j$ functions \cite{HochOst10}] \label{def:phi}
Set
\[
\varphi_0(z ) := \e^{z}\qquad\text{and}\qquad \varphi_k(z) := \int_0^1 \e^{(1-\theta)z} \frac{\theta^{k-1}}{(k-1)!}\dd \theta, \quad k \geq 1
\]
such that in particular
\begin{equation*}
\begin{aligned}
\varphi_0(z ) = \e^{z}, \qquad \varphi_1(z) = \frac{\mathrm{e}^{z} - 1}{z}, \qquad \varphi_2(z) = \frac{\varphi_0(z) - \varphi_1(z)}{z}.
\end{aligned}
\end{equation*}
\end{defn}
In the following we assume local-wellposedness (LWP) of \eqref{eq:ua1} in $H^r$.
\begin{ass}
Fix $r>d/2$ and assume that there exists a $T_r = T>0$ such that the solutions $(\ua(t),\va(t))$ of \eqref{eq:ua1} satisfy
\begin{align*}
\sup_{0 \leq t \leq T} \Vert \ua(t) \Vert_{r}+ \Vert \va(t) \Vert_{r} \leq M
\end{align*}
uniformly in $c$.
\end{ass}
\begin{rem} The previous assumption holds under the following condition on the initial data
$$ \Vert z(0) \Vert_{r} + \Vert c^{-1}\nab^{-1} z'(0)\Vert_{r} \leq M_0$$
where $M_0$ does not depend on $c$ as can be easily proved from the formulation \eqref{du0}. 
\end{rem}

\section{A first-order uniformly accurate scheme}\label{sec:scheme1}
In this section we derive a  first-order  exponential-type integration scheme for the solutions $(\ua,\va)$ of \eqref{eq:ua1} which allows \emph{first-order uniform time-convergence with respect to $c$}. The construction is thereby based on Duhamel's formula \eqref{du0} and the essential estimates in Lemma \ref{lem:bAc}, \ref{lem:expo} and \ref{lem:upc}.   For the derivation we will for simplicity assume that $z$ is real, which (by Remark \ref{rem:realz}) implies that $u = v$ such that system \eqref{eq:ua1} reduces to
\begin{equation}\label{eq:ua}
i\partial_t \ua=- \Ac \ua+ \frac{1}{8} c \nab^{-1} \mathrm{e}^{-ic^2t} \left( \e^{ic^2t} \ua+ \e^{-ic^2t} \overline{\ua}\right)^3
\end{equation}
with mild-solutions
\begin{equation}\label{du}
\begin{aligned}
\ua(t_n+\tau)&  = \e^{i \tau \Ac} \ua(t_n)\\ & - \frac{i}{8} c \nab^{-1} \int_0^\tau \e^{i(\tau-s) \Ac} \e^{-i c^2 (t_n+s)} \left(
\e^{ic^2(t_n+s)} \ua(t_n+s) + \e^{-ic^2(t_n+s)} \overline{\ua}(t_n+s)\right)^3 \mathrm{d}s.
\end{aligned}
\end{equation}

\subsection{Construction}
In order to derive a  first-order scheme, we need to impose additional regularity  assumptions on the exact solution $\ua(t)$ of \eqref{eq:ua}.
\begin{ass}\label{ass:reg1}
Fix $r>d/2$ and assume that $\ua \in \mathcal{C}([0,T];H^{r+2}(\mathbb{T}^d))$ and  in particular
\begin{align*}
\sup_{0 \leq t \leq T} \Vert \ua(t) \Vert_{r+2}\leq M_2 \quad \text{uniformly in $c$}.
\end{align*}
\end{ass}

Applying Lemma \ref{lem:expo} and Lemma \ref{lem:upc} in \eqref{du} allows us the following expansion
\begin{equation}\label{du1}
\begin{aligned}
\ua(t_n +\tau) & = \e^{i \tau \Ac} \ua(t_n) - \frac{i}{8} c \nab^{-1} \e^{i \tau \Ac} \int_0^\tau  \e^{-i c^2 (t_n+s)} \left(
 \e^{ic^2(t_n+s)} \ua(t_n)+ \e^{-ic^2(t_n+s)} \overline{\ua}(t_n)\right) ^3 \mathrm{d}s
 \\&+ \mathcal{R}(\tau,t_n,\ua),
\end{aligned}
\end{equation}
where the remainder $ \mathcal{R}(\tau,t_n,\ua)$ satisfies thanks to the bounds \eqref{approx1}, \eqref{approx2} and \eqref{cnabm} that
\begin{equation}\label{r1p}
\Vert  \mathcal{R}(\tau,t_n,\ua)\Vert_r \leq \tau^2 k_r(M_2),
\end{equation}
for some constant $k_r(M_2)$ which depends on $M_2$ (see Assumption \ref{ass:reg1}) and $r$, but is independent of $c$. Solving the integral in \eqref{du1} (in particular, integrating the highly-oscillatory phases $\mathrm{exp}(\pm i l c^2 s)$ exactly) furthermore yields  by adding and subtracting the term $  \tau \frac{3i}{8} \e^{i \tau \Ac} \vert \ua(t_n)\vert^2 \ua(t_n)$ (see  Remark \ref{rem:limitLie} below for the purpose of this manipulation)  that
\begin{equation}\label{du2p}
\begin{aligned}
&  \ua(t_n+\tau)
 = \e^{i \tau \Ac} \Big( 1  - \tau \frac{3i}{8} \vert \ua(t_n)\vert^2 \Big)\ua(t_n)   - \tau \frac{3i}{8} \left(c \nab^{-1}-1\right) \e^{i \tau \Ac} \vert \ua(t_n)\vert^2 \ua(t_n) 
\\&- \tau \frac{i}{8} c \nab^{-1} \e^{i \tau \Ac} \Big\{ \e^{2ic^2t_n}\varphi_1(2ic^2 \tau) \ua^3(t_n)
 +\e^{-2ic^2t_n} \varphi_1(-2ic^2\tau) 3\vert \ua(t_n)\vert^2 \overline{\ua}(t_n) 
\\& +\e^{-4ic^2t_n} \varphi_1(-4ic^2\tau) \overline{\ua}^3(t_n)
\Big\} + \mathcal{R}(\tau,t_n,\ua)
\end{aligned}
\end{equation}
with $\varphi_1$ given in Definition \ref{def:phi}.

As the operator $\e^{it \Ac}$ is a linear isometry in $H^r$ and by Taylor series expansion 
$
\vert 1-x - \e^{-x} \vert = \mathcal{O}(x^2)
$
we obtain for $r>d/2$  that
\begin{equation}\label{exiL}
\begin{aligned}
\left \Vert \e^{i \tau \Ac} \Big(  1  - \tau \frac{3i}{8}  \vert \ua(t_n)\vert^2 \ua(t_n)\Big)- \e^{i \tau \Ac}  \mathrm{e}^{-\tau \frac{3i}{8}  \vert \ua(t_n)\vert^2  } \ua(t_n) \right\Vert_r  
 \leq k_{r} 3 \tau^2 \Vert  \ua(t_n)\Vert_r^3
\end{aligned}
\end{equation}
for some constant $k_r$ independent of $c$. 

The bound in \eqref{exiL} allows us to express \eqref{du2p}  as follows
\begin{equation}\label{du2}
\begin{aligned}
  \ua(t_n+\tau)
& = \e^{i \tau \Ac} \mathrm{e}^{- \tau \frac{3i}{8} \vert \ua(t_n)\vert^2}\ua(t_n)   - \tau \frac{3i}{8} \left(c \nab^{-1}-1\right) \e^{i \tau \Ac}\vert \ua(t_n)\vert^2 \ua(t_n) 
\\&- \tau \frac{i}{8} c \nab^{-1} \e^{i \tau \Ac} \Big\{ \e^{2ic^2t_n}\varphi_1(2ic^2 \tau) \ua^3(t_n) 
 +\e^{-2ic^2t_n} \varphi_1(-2ic^2\tau) 3 \vert \ua(t_n)\vert^2  \overline{\ua}(t_n) 
\\& +\e^{-4ic^2t_n} \varphi_1(-4ic^2\tau) \overline{\ua}^3(t_n)
\Big\} + \mathcal{R}(\tau,t_n,\ua),
\end{aligned}
\end{equation}
where the remainder $  \mathcal{R}(\tau,t_n,\ua)$ satisfies thanks to \eqref{r1p}  and \eqref{exiL} that
\begin{equation}\label{rem1}
\Vert  \mathcal{R}(\tau,t_n,\ua)\Vert_r  \leq \tau^2 k_r(M_2),
\end{equation}
for some constant $k_r(M_2)$ which depends on $M_2$ (see Assumption \ref{ass:reg1}) and $r$, but is independent of $c$.

The expansion \eqref{du2}  of the exact solution $\ua(t)$ builds the basis of our numerical scheme: As a numerical approximation to the exact solution $\ua(t)$ at time $t_{n+1} = t_n + \tau$ we choose the exponential-type integration scheme
\begin{equation}\label{scheme100}
\begin{aligned}
\ua^{n+1}  & =\e^{i \tau \Ac} \mathrm{e}^{- \tau \frac{3i}{8} \vert \ua^n\vert^2} \ua^n  -\tau \frac{3i}{8} \left(c \nab^{-1}-1\right) \e^{i \tau \Ac}  \vert \ua^n\vert^2\ua^n
\\&- \tau \frac{i}{8} c \nab^{-1}\e^{i \tau \Ac}\Big\{
\e^{2ic^2t_n} \varphi_1(2ic^2\tau)( \ua^n)^3
+  \e^{-2ic^2t_n} \varphi_1(-2ic^2\tau)3\vert \ua^n\vert^2\overline{\ua^n}\\&\qquad\qquad\qquad\qquad
+  \e^{-4ic^2t_n} \varphi_1(-4ic^2\tau) ( \overline{\ua^{n}})^3\Big\}\\
 \ua^0  & = z(0) -ic^{-1}\nab^{-1} z'(0) 
\end{aligned}
\end{equation}
with $\varphi_1$ given in Definition \ref{def:phi}. Note that the definition of the initial value $\ua^0$ follows from \eqref{eq:BCc}. 

 For complex-valued functions $z$ (i.e., for $u\not \equiv v$) we similarly derive the exponential-type integration scheme 
\begin{equation}\label{scheme1}
\begin{aligned}
 \ua^{n+1}  &=\e^{i \tau \Ac} \mathrm{e}^{- \tau \frac{i}{8}  \big( \vert \ua^n\vert^2+2 \vert \va^n\vert^2\big)} \ua^n  -\tau \frac{i}{8} \left(c \nab^{-1}-1\right) \e^{i \tau \Ac} \big( \vert \ua^n\vert^2+2 \vert \va^n\vert^2\big)\ua^n
\\&- \tau \frac{i}{8} c \nab^{-1}\e^{i \tau \Ac}\Big\{
\e^{2ic^2t_n} \varphi_1(2ic^2\tau)( \ua^n)^2 \va^n
+  \e^{-2ic^2t_n} \varphi_1(-2ic^2\tau)  \big(2 \vert \ua^n\vert^2+ \vert \va^n\vert^2\big)\overline{\va^n}
\\&\qquad\qquad\qquad\qquad+  \e^{-4ic^2t_n} \varphi_1(-4ic^2\tau) ( \overline{\va^{n}})^2\overline{\ua^n}\Big\}\\
 \ua^0  &= z(0) -ic^{-1}\nab^{-1} z'(0),
\end{aligned}
\end{equation}
where  the scheme in $\va^{n+1}$ is obtained by replacing $\ua^n \leftrightarrow \va^n$ on the right-hand side of \eqref{scheme1} with initial value $\va^0 = \overline{z(0)} - i c^{-1}\nab^{-1}\overline{z'(0)}$ (see \eqref{eq:BCc}).

\begin{rem}[Practical implementation] To reduce the computational effort we may express the first-order scheme \eqref{scheme1} in its equivalent form
\begin{equation*}\label{scheme1Pra}
\begin{aligned}
&\ua^{n+1} = \e^{i \tau \Ac} \left( \mathrm{e}^{-\tau \frac{i}{8} \big( \vert \ua^n\vert^2+2 \vert \va^n\vert^2\big) } \ua^n +  \tau \frac{i}{8}  \big( \vert \ua^n\vert^2+2 \vert \va^n\vert^2\big) \ua^n\right)- \frac{i\tau}{8} c \nab^{-1}\e^{i \tau \Ac}\Big\{
\big( \vert \ua^n\vert^2+2 \vert \va^n\vert^2\big) \ua^n
 \\&
+\e^{2ic^2t_n} \varphi_1(2ic^2\tau)( \ua^n)^2 \va^n
+  \e^{-2ic^2t_n} \varphi_1(-2ic^2\tau)  \big(2 \vert \ua^n\vert^2+ \vert \va^n\vert^2\big)\overline{\va^n}
+  \e^{-4ic^2t_n} \varphi_1(-4ic^2\tau) ( \overline{\va^{n}})^2\overline{\ua^n} \Big\}\\
&\quad \ua^0   = z(0) -ic^{-1}\nab^{-1} z'(0) 
\end{aligned}
\end{equation*}
which after a Fourier pseudo-spectral space discretization only requires the usage of two Fast Fourier transforms (and its corresponding inverse counter parts) instead of three.

\end{rem}

In Section \ref{sec:con1} below we prove that the exponential-type integration scheme \eqref{scheme1} is first-order convergent uniformly in $c$ for sufficiently smooth solutions. Furthermore, we give a fractional convergence result under weaker regularity assumptions and analyze its behavior in the non-relativistic limit regime $c \to \infty$. In Section \ref{sec:limit1} we give some simplifications in the latter regime.

%
%\red
%Scheme 2
%\begin{equation}\label{scheme2}
%\begin{aligned}
%\ua^{n+1} & =\e^{i \tau \Ac}  \mathrm{e}^{-i \tau \frac{3}{8} \vert \ua^n\vert^2 } \ua^n  -i \tau \frac{3}{8} \left(c \nab^{-1}-1\right) \left( \vert \ua^n\vert^2 \ua^n\right)
%\\&- \frac{i\tau}{8} c \nab^{-1}\Big\{
%\e^{2ic^2t_n} \varphi_1(2ic^2\tau)( \ua^n)^3
%+ 3 \e^{-2ic^2t_n} \varphi_1(-2ic^2\tau) \vert \ua^n\vert^2 \overline{\ua}^n
%\\&+  \e^{-4ic^2t_n} \varphi_1(-4ic^2\tau) ( \overline{\ua^{n}})^3\Big\}.
%\end{aligned}
%\end{equation}
%
% Scheme 3
%\begin{equation}\label{scheme3}
%\begin{aligned}
%\ua^{n+1} & =\e^{i \tau \Ac}  \mathrm{e}^{-i \tau \frac{3}{8} \vert \ua^n\vert^2 } \ua^n  -i \tau \frac{3}{8} \left(c \nab^{-1}-1\right) \left( \vert \ua^n\vert^2 \ua^n\right)
%\\&- \frac{i\tau}{8} c \nab^{-1}\Big\{
%\e^{2ic^2 t_n}\varphi_1(2i\tau\nab^2) (\ua^n)^3 + 3\e^{-2ic^2t_n} \varphi_1(i\tau(-2c^2-\Ac)) \vert \ua^n\vert^2 \overline{\ua^n}\\&
%+\e^{-4ic^2t_n} \varphi_1(i\tau(-4c^2-\Ac))(\overline{\ua^n})^3
%\Big\}.
%\end{aligned}
%\end{equation}
%

\subsection{Convergence analysis}\label{sec:con1}
The exponential-type integration scheme \eqref{scheme1} converges (by construction) with first-order in time uniformly with respect to $c$, see Theorem \ref{them:con1}. Furthermore, a fractional convergence bound holds true for less regular solutions, see Theorem \ref{them:con1Frac}. In particular, in the limit $c \to \infty$ the scheme converges to the classical Lie splitting applied to the nonlinear Schr\"odinger limit system, see Lemma \ref{rem:limit1}.

\begin{thm}[Convergence bound for the first-order scheme] \label{them:con1} Fix $r>d/2$ and assume that
\begin{align}\label{eq:urged1}
 \Vert z(0) \Vert_{r+2} + \Vert c^{-1}\nab^{-1} z'(0)\Vert_{r+2} \leq M_2
\end{align}
uniformly in $c$.  For $(\ua^{n},\va^n)$ defined in \eqref{scheme1} we set
\[
z^{n} := \frac12\left(  \e^{ic^2 t_n} \ua^{n} +  \e^{-ic^2 t_n} \overline{\va^{n}}\right).
\]
Then, there exists a $T_r >0$ and $\tau_0>0$ such that for all $\tau \leq \tau_0$ and $t_n \leq T_r$ we have for all $c >0$ that
\begin{align*}\label{glob1}
\left \Vert z(t_{n}) - z^{n} \right\Vert_r \leq \tau K_{1,r,M_2} \mathrm{e}^{t_n K_{2,r,M}} \leq \tau K^\ast_{r,M,M_2,t_n},
\end{align*}
where the constants $K_{1,r,M_2},K_{2,r,M}$ and $K^\ast_{r,M,M_2,t_n}$ can be chosen independently of $c$.
\end{thm}
\begin{proof} Fix $r>d/2$. First note that the regularity assumption on the initial data in \eqref{eq:urged1} implies the regularity Assumption \ref{ass:reg1} on $(\ua,\va)$, i.e., there exists a $T_r>0$ such that
\[
\sup_{0 \leq t \leq T_r} \Vert \ua(t)\Vert_{r+2} + \Vert \va(t) \Vert_{r+2} \leq k(M_2)
\]
for some constant $k$ that depends on $M_2$ and $T_r$, but can be chosen independently of $c$.\\

In the following let  $(\phi^t_{u_{ \ast}},\phi^t_{v_{ \ast}})$ denote the exact flow of \eqref{eq:ua1} and let $(\Phi^\tau_{u_{\ast}},\Phi^\tau_{v_{\ast}})$ denote the numerical flow defined in \eqref{scheme1}, i.e.,
\[
\ua(t_{n+1}) = \phi^\tau_{\ua}(\ua(t_n),\va(t_n)), \qquad \ua^{n+1}= \Phi^\tau_{\ua}(\ua^n,\va^n)
\]
and a similar formula for the functions $\va(t_n)$ and $\va^n$. 
This allows us to split the global error as follows
\begin{equation}\label{glob0}
\begin{aligned}
\ua(t_{n+1}) - \ua^{n+1}& = \phi^\tau_{\ua}(\ua(t_n),\va(t_n)) - \Phi^\tau_{\ua}(\ua^n,\va^n)\\
&= \Phi^\tau_{\ua}(\ua(t_n),\va(t_n)) - \Phi^\tau_{\ua}(\ua^n,\va^n) + 
 \phi^\tau_{\ua}(\ua(t_n),\va(t_n))
- \Phi^\tau_{\ua}(\ua(t_n),\va(t_n)).
\end{aligned}
\end{equation}

\emph{Local error bound:} With the aid of \eqref{rem1} we have by the expansion of the exact solution in \eqref{du2} and the definition of the numerical scheme \eqref{scheme1} that
\begin{equation}\label{local1}
\Vert  \phi^\tau_{\ua}(\ua(t_n),\va(t_n))
- \Phi^\tau_{\ua}(\ua(t_n),\va(t_n)) \Vert_r = \Vert  \mathcal{R}(\tau,t_n,\ua,\va)\Vert_r \leq \tau^2 k_r(M_2)
\end{equation}
for some constant $k_r$ which depends on $M_2$ and $r$, but can be chosen independently of $c$.

\emph{Stability bound:}  Note that for all $l \in \mathbb{Z}$  we have that
\[
\Vert \varphi_1(i \tau c^2 l) \Vert_r \leq 2.
\]
Thus, as $\mathrm{e}^{i t \Ac}$ is a linear isometry for all $t \in \mathbb{R}$ we obtain together with the bound \eqref{cnabm} that as long as $\Vert \ua^n\Vert_r \leq 2M$ and $\Vert u(t_n)\Vert_r \leq M$ we have that
\begin{equation}\label{stab1}
\Vert\Phi^\tau_{\ua}(\ua(t_n),\va(t_n)) - \Phi^\tau_{\ua}(\ua^n,\va^n)\Vert_r \leq \Vert \ua(t_n) - \ua^n\Vert_r +\tau K_{r,M} \left(\Vert \ua(t_n) - \ua^n\Vert_r + \Vert \va(t_n) - \va^n\Vert_r\right),
\end{equation}
where the constant $K_{r,M}$ depends on $r$ and $M$, but can be chosen independently of $c$.

\emph{Global error bound:} Plugging the stability bound \eqref{stab1}  as well as the local error bound \eqref{local1} into \eqref{glob0} yields by a bootstrap argument that
\begin{align}\label{conus}
\left \Vert \ua(t_{n}) - \ua^{n} \right\Vert_r \leq \tau K_{1,r,M_2} \mathrm{e}^{t_n K_{2,r,M}},
\end{align}
where the constants are uniform in $c$. A similar bound holds for the difference $\va(t_n) -\va^n$. This implies first-order convergence of $(\ua^n,\va^n)$ towards $(\ua(t_n),\va(t_n))$ uniformly in $c$.

Furthermore, by \eqref{eq:zuv} and \eqref{psi} we have that
\begin{align*}
 \Vert z(t_n) - z^n\Vert_r & = \textstyle \left \Vert\frac12 \big( u(t_n) + \overline{v(t_n)}\big) -\frac12 \big(\e^{ic^2 t_n} \ua^n + \e^{-ic^2t_n} \overline{\va^n}\big)\right\Vert\\
&  \leq \Vert \e^{ic^2t_n} (\ua(t_n)-\ua^n)\Vert_r + \Vert  \e^{ic^2t_n} (\va(t_n)-\va^n)\Vert_r   \\
& =  \Vert \ua(t_n)-\ua^n\Vert_r + \Vert  \va(t_n)-\va^n\Vert_r .
\end{align*}
Together with the bound in \eqref{conus} this completes the proof.
\end{proof}

\begin{rem}
Note that the regularity assumption \eqref{eq:urged1} is always satisfied for initial values
\[
z(0,x) = \varphi(x),\qquad \partial_t z(0,x) = c^2 \gamma(x) \qquad \text{with} \quad \varphi, \gamma \in H^{r+2}
\]
as then thanks to \eqref{cnabm} we have
\[
\left \Vert c^{-1}\nab^{-1}z'(0) \right\Vert_r = \left \Vert c\nab^{-1} \gamma\right\Vert_r \leq \Vert \gamma \Vert_r.
\]
\end{rem}

Under weaker regularity assumptions on the exact solution we obtain \emph{uniform fractional convergence} of the formally first-order scheme \eqref{scheme1}.
\begin{thm}[Fractional convergence bound for the first-order scheme] \label{them:con1Frac}
Fix $r>d/2$ and assume that for some $0< \gamma \leq 1$
\begin{align}\label{eq:urged}
 \Vert z(0) \Vert_{r+2\gamma} + \Vert c^{-1}\nab^{-1} z'(0)\Vert_{r+2\gamma}  \leq M_{2\gamma}
\end{align}
uniformly in $c$. For $(\ua^{n},\va^n)$ defined in \eqref{scheme1} we set
\[
z^{n} := \frac12\left(  \e^{ic^2 t_n} \ua^{n} +  \e^{-ic^2 t_n} \overline{\va^{n}}\right).
\]
Then, there exists a $T_r>0$ and $\tau_0>0$ such that for all $\tau \leq \tau_0$ and $t_n \leq T_r$ we have for all $c >0$ that
\begin{equation*}
\begin{aligned}
\left \Vert z(t_{n}) - z^{n} \right\Vert_r \leq \tau^\gamma K_{1,r,M_{2\gamma}} \mathrm{e}^{t_n K_{2,r,M}} \leq \tau^\gamma K^\ast_{r,M,M_{2\gamma},t_n},
\end{aligned}
\end{equation*}
where the constants $K_{1,r,M_{2\gamma}},K_{2,r,M}$ and $K^\ast_{r,M,M_{2\gamma},t_n}$ can be chosen independently of $c$.
\end{thm}
\begin{proof}
The proof follows the line of argumentation to the proof of Theorem \ref{them:con1} using ``fractional estimates'' of the operator $\Ac$.

%Fix $r>d/2$ and $0 < \gamma \leq 1$. First note that similarly to Lemma \ref{lem:boundAc} we obtain that
%\[
%\Vert \Ac^\gamma f \Vert_r \leq 2^{-\gamma} \Vert f \Vert_{r+2\gamma}.
%\]
%Furthermore, as $\left \vert\e^{ix}-1 \right \vert \leq 2 \vert x \vert^\gamma$ for all $x \in \mathbb{R}$ we have that
%\begin{align*}
% \left \Vert \left(
%\e^{-i t \Ac}-1 \right)f \right\Vert_r \leq 2 \Vert \Ac^\gamma f \Vert_r \leq 2^{1-\gamma} \vert t \vert^\gamma \Vert f \Vert_{r+2\gamma}.
%\end{align*}
%Thus in particular, Duhamel's formula \eqref{du0} together with the bound in \eqref{cnabm} yields for $r>d/2$ that
%\begin{align*}
% \Vert \ua(t_n+s) - \ua(t_n) \Vert_r +  \Vert \va(t_n+s) - \va(t_n) \Vert_r\\ \leq \vert s \vert^\gamma \big( \Vert \Ac^\gamma \ua(t_n)\Vert_r + \Vert \Ac^\gamma \va(t_n)\Vert\big) + \vert s \vert (1+M_0)^3.
%\end{align*}
%The above bounds yield the corresponding fractional estimates of Lemma \ref{lem:bAc}, \ref{lem:expo} and \ref{lem:upc}. With these fractional error bounds at hand, the proof then follows the line of argumentation to the proof of Theorem \ref{them:con1}.
\end{proof}

Next we point out an interesting observation: For sufficiently smooth solutions the exponential-type integration scheme \eqref{scheme1} converges in the limit $c \to \infty$ to the classical Lie splitting of  the corresponding nonlinear Schr\"odinger limit \eqref{NLSlimit}.

\begin{rem}[Approximation in the non relativistic limit $c \to \infty$]\label{rem:limit1}
The exponential-type integration scheme \eqref{scheme1} corresponds for sufficiently smooth solutions in the limit $(\ua^n,\va^n) \stackrel{c\to \infty}{\longrightarrow} (u_{\ast,\infty}^n,v_{\ast,\infty}^n)$, essentially to the Lie Splitting  (\cite{Lubich08,Faou12})
\begin{equation}\label{limitLie}
\begin{aligned}
u_{\ast, \infty}^{n+1}  &=\e^{-i \tau \frac{\Delta}{2}}  \mathrm{e}^{-i \tau \frac{1}{8} \big( \vert u_{\ast, \infty}^n\vert^2 +2  \vert v_{\ast, \infty}^n\vert^2 \big)}u_{\ast, \infty}^n,\qquad u_{\ast,\infty}^0 = \varphi - i \gamma,
\\
v_{\ast, \infty}^{n+1}  &=\e^{-i \tau \frac{\Delta}{2}}  \mathrm{e}^{-i \tau \frac{1}{8} \big( \vert v_{\ast, \infty}^n\vert^2 +2  \vert u_{\ast, \infty}^n\vert^2 \big)}v_{\ast, \infty}^n,\qquad v_{\ast,\infty}^0 = \overline{\varphi} - i \overline{\gamma}
\end{aligned}
\end{equation}
applied to the cubic nonlinear Schr\"odinger system \eqref{NLSlimit} which is the limit system of the Klein-Gordon equation \eqref{eq:kgr} for $c \to \infty$ with initial values
\begin{align*}
z(0) \stackrel{c\to\infty}{\longrightarrow} \gamma \quad \text{and}\quad c^{-1}\nab^{-1} z'(0) \stackrel{c\to\infty}{\longrightarrow} \varphi.
\end{align*}

More precisely, the following Lemma holds.
\end{rem}

\begin{lem}\label{rem:limit1} 
Fix $r>d/2$ and let $ 0  < \delta \leq 2$. Assume that
\begin{equation}\label{regass:limit1}
\Vert z(0) \Vert_{r+2\delta+\varepsilon} +\Vert c^{-1}\nab^{-1} z'(0)\Vert_{r+2\delta+\varepsilon} \leq M_{2\delta+\varepsilon}
\end{equation}
for some $\varepsilon >0$ uniformly in $c$ and let the initial value approximation (there exist functions $\varphi,\gamma$ such that)
\begin{align}\label{limitIn}
\Vert z(0)- \gamma\Vert_r + \Vert c^{-1}\nab^{-1} z'(0) - \varphi\Vert_r \leq k_r c^{-\delta}
\end{align}
hold for some constant $k_r$ independent of $c$.

Then, there exists a $T>0$ and $\tau_0>0$ such that for all $\tau \leq \tau_0$ the difference of the first-order scheme \eqref{scheme1} for system \eqref{eq:ua1} and the Lie splitting \eqref{limitLie} for the limit Schr\"odinger equation \eqref{NLSlimit} satisfies for $ t_n \leq T$ and all $c>0$ with
\begin{equation}\label{ctau1}
\tau c^{2-\delta} \geq 1
\end{equation}
that
\begin{align*}
\Vert \ua^n- u_{\ast, \infty}^{n} \Vert_r + \Vert \va^n- v_{\ast, \infty}^{n} \Vert_r \leq c^{-\delta} k_r(M_{2\delta+\varepsilon},T)
\end{align*}
for some constant $k_{r}$ that depends on $M_{2\delta+\varepsilon}$ and $T$, but is independent of $c$.
\end{lem}
\begin{proof}
In the following fix $r>d/2$, $ 0  <  \delta \leq 2$ and $\varepsilon>0$:

\emph{1. Initial value approximation:}  Thanks to \eqref{limitIn} we have by the definition of the initial value $\ua^0$ in \eqref{scheme1}, respectively, $u_{\ast,\infty}^0$ in \eqref{limitLie} that
\[
\Vert \ua^0 - u_{\ast,\infty}^0 \Vert_r = \Vert z(0) -ic^{-1}\nab^{-1} z'(0) - (\varphi-i\gamma)\Vert_r \leq k_r c^{-\delta}
\]
for some constant $k_r$ independent of $c$. A similar bound holds for  $\va^0 - v_{\ast,\infty}^0$. 

\emph{2. Regularity of the numerical solutions $(\ua^n,\va^n)$:}  Thanks to the  regularity assumption \eqref{regass:limit1} we have by  Theorem \ref{them:con1Frac} that there exists a $T>0$ and $\tau_0>0$ such that for all $\tau \leq \tau_0$ we have
\begin{equation}\label{regNum1}
\Vert \ua^n\Vert_{r+2\delta} + \Vert \va^n \Vert_{r+2\delta} \leq m_{2\delta}
\end{equation}
as long as $t_n \leq T$ for some constant $m_{2\delta}$ depending on $M_{2\delta+\varepsilon}$ and $T$, but not on $c$.

\emph{3. Regularity of the numerical solutions $(u_{\ast,\infty}^n,v_{\ast,\infty}^n)$:}  Thanks to the  regularity assumption \eqref{regass:limit1} we have by \eqref{limitIn} and the global first-order convergence result of the Lie splitting for semilinear Schr\"odinger equations (see for instance  \cite{Faou12,Lubich08}) that there exists a $T>0$ and $\tau_0>0$ such that for all $\tau \leq \tau_0$ we have\begin{equation}\label{regNum2}
\Vert u_{\ast,\infty}^n\Vert_{r} + \Vert v_{\ast,\infty}^n \Vert_{r} \leq m_{0}
\end{equation} 
as long as $t_n \leq T$ for some constant $m_{0}$ depending on $M_{r}$ and $T$, but not on $c$.

\emph{4. Approximations:} 
Using the following bounds, $\gamma > 1$
\begin{equation}
\label{kater}
\left|Ê\sqrt{1 + x^2} - 1 - \frac12 x^2\right|Ê \leq x^{2\gamma}  \quad \mbox{and} \quad \left|Ê\frac{1}{\sqrt{1 + x^2}¿} - 1\right| \leq x^{2\gamma - 2}, 
\end{equation}
together with the Definition of $\varphi_1$ (see Definition \ref{def:phi}) we have for every $f \in H^{r+2+2\delta}$, 
\begin{align}\label{boundOpc1}
\big\Vert \left(\Ac + \textstyle\frac{\Delta}{2}\right) f \big\Vert_r + \big  \Vert \left(c \nab^{-1}-1\right) f \big \Vert_{r+2} + \big \Vert \varphi_1(ilc^2\tau)f \big\Vert_{r+2+\delta}   \leq k_{r} c^{-\delta} \Vert f \Vert_{r+2+2\delta}
\end{align}
for $l = \pm 2, -4$ and for some constant $k_{r}$ independent of $c$, where we used \eqref{ctau1} for the last estimate. 

\emph{5. Difference of the numerical solutions: } Thanks to the a priori regularity of the numerical solutions \eqref{regNum1} and \eqref{regNum2}  we obtain with the aid of \eqref{boundOpc1} under  assumption  \eqref{ctau1} for the difference $\ua^n- u_{\ast,\infty}^n$ that
\begin{equation}\label{scheme1Ex}
\begin{aligned}
\Vert \ua^{n+1} - u_{\ast,\infty}^{n+1} \Vert_r & \leq \big(1+ \tau k(m_{0})\big) \Vert \ua^n - u_{\ast,\infty}^n\Vert_r 
+ (c^{-2+\delta}+\tau) c^{-\delta} k(m_{2\delta})\\
& \leq \big(1+ \tau k(m_{0})\big)  \Vert \ua^n - u_{\ast,\infty}^n\Vert_r 
+ 2 \tau c^{-\delta} k(m_{2\delta})
\end{aligned}
\end{equation}
and a similar bound on $\va^n- v_{\ast,\infty}^n$. Solving the recursion yields the assertion.
\end{proof}

\subsection{Simplifications in the ``weakly to strongly  non-relativistic limit regime''}\label{sec:limit1}
In the `` strongly non-relativistic limit regime'', i.e., for large values of $c$, we may simplify the first-order scheme \eqref{scheme1} and nevertheless obtain a well suited, first-order approximation to $(\ua,\va)$ in \eqref{eq:ua1}.
\begin{rem} Note that for $l = \pm 2, -4$ we have (see Definition \ref{def:phi})
\[
\left \Vert \tau  \varphi_1(i l c^2 \tau) \right \Vert_r \leq 2 c^{-2} .
\]
Furthermore, \eqref{boundOpc1} yields that
\[
\Vert \left (c\nab^{-1} - 1\right) \ua(t) \Vert_r \leq c^{-2} k_r \Vert \ua(t)\Vert_{r+2}
\]
for some constant $k_r$ independent of $c$.

Thus, for sufficiently large values of $c$, more precisely if
\[
\tau c > 1
\]
and under the same regularity assumption \eqref{eq:urged} we may take instead of \eqref{scheme1} the scheme
\begin{align*}
u_{\ast,c>\tau}^{n+1}  & =\e^{i \tau \Ac}  \mathrm{e}^{-i \tau\frac{1}{8} \big( \vert u_{\ast,c>\tau}^n\vert^2 +2\vert v_{\ast,c>\tau}^n\vert^2 \big)} u_{\ast,c>\tau}^n \\
v_{\ast,c>\tau}^{n+1}  & =\e^{i \tau \Ac}  \mathrm{e}^{-i\tau\frac{1}{8} \big( \vert v_{\ast,c>\tau}^n\vert^2 +2\vert u_{\ast,c>\tau}^n\vert^2 \big)} v_{\ast,c>\tau}^n  
\end{align*}
as a first-order numerical approximation to $(\ua(t_{n+1}),\va(t_{n+1}))$ in \eqref{eq:ua1}.
\end{rem}

However, note that in the strongly non-relativistic limit regime (such that in particular $c \tau > 1$) we may immediately take the Lie splitting scheme proposed in \cite{FS13} as a suitable first-order approximation to \eqref{eq:ua1}  thanks to the following observation:
\begin{rem}[Limit scheme \cite{FS13}]\label{rem:limitLie}
For sufficiently large values of $c$ and sufficiently smooth solutions, more precisely, if
\[
 \Vert z(0) \Vert_{r+2} + \Vert c^{-1}\nab^{-1} z'(0)\Vert_{r+2}\leq M_{2}\quad \text{and}\quad \tau c > 1
\]
 the classical Lie splitting (see  \cite{Lubich08,Faou12}) for the nonlinear Schr\"odinger limit equation \eqref{NLSlimit}, namely,
\begin{equation}\label{limitscheme1}
\begin{aligned}
u_{\ast,\infty}^{n+1}  &=\e^{-i \tau \frac{1}{2}\Delta }  \mathrm{e}^{-i \tau \frac{1}{8}\big( \vert u_{\ast,\infty}^n\vert^2 +2\vert v_{\ast,\infty}^n\vert^2\big) } u_{\ast,\infty}^n\\
v_{\ast,\infty}^{n+1}  &=\e^{-i \tau \frac{1}{2}\Delta }  \mathrm{e}^{-i \tau \frac{1}{8}\big( \vert v_{\ast,\infty}^n\vert^2 +2\vert u_{\ast,\infty}^n\vert^2\big) } v_{\ast,\infty}^n
\end{aligned}
\end{equation}
 yields as a first-order numerical approximation to $(\ua(t_{n+1}),\va(t_{n+1}))$ in \eqref{eq:ua1}.

This assertion follows from \cite{FS13} thanks to the approximation
\begin{align*}
\Vert \ua(t_n) -  u_{\ast,\infty}^{n} \Vert_r & \leq \Vert \ua(t_n) - u_{\ast,\infty}(t_n)\Vert_r  + \Vert u_{\ast,\infty}(t_n) - u_{\ast,\infty}^{n} \Vert_r = \mathcal{O}\big( c^{-1}+\tau\big)
\end{align*}
and the similar bound on $\va(t_n) -  v_{\ast,\infty}^{n}$.
\end{rem}

\section{A second-order uniformly accurate scheme}\label{sec:scheme2}
In this section we derive a second-order exponential-type integration scheme for the solutions $(\ua,\va)$ of \eqref{eq:ua1} which allows \emph{second-order uniform time-convergence with respect to $c$}.  For notational simplicity we  again assume that $z$ is real, which reduces the coupled system \eqref{eq:ua1} to equation  \eqref{eq:ua} with mild-solutions \eqref{du} (see also Remark \ref{rem:realz}).

The construction of the second-order scheme is again based on Duhamel's formula \eqref{du} and the essential estimates in Lemma \ref{lem:bAc}, \ref{lem:expo} and \ref{lem:upc}.  However, the second-order approximation is much more involved due to the fact that 
$$
 \ua'(t) = \mathcal{O}(1), \quad \text{but}\quad  \ua''(t) = \mathcal{O}(c^2).
$$
The latter observation prevents us from simply applying the higher-order Taylor series expansion
$$
\ua(t_n+s) = \ua(t_n) + s \ua'(t_n) + \mathcal{O}\big(s^2 \ua''(t_n+\xi)\big)
$$
in Duhamel's formula \eqref{du} as this would lead to the ``classical'' $c-$dependent error at order $\mathcal{O}(\tau^2 c^2)$. Therefore we need to carry out a much more careful frequency analysis  by iterating Duhamel's formula \eqref{du} twice and controlling the appearing highly-oscillatory terms $\mathrm{e}^{\pm i c^2 t}$ and their interactions $\mathrm{e}^{il c^2 t}$ ($l \in \mathbb{Z}$) precisely.

\subsection{Construction of a second-order uniformly accurate scheme}
In Section \ref{submerge} we state the necessary regularity assumptions on the solution $\ua$ and derive two useful expansions. In Section \ref{sub:PLD} we collect some useful lemmata on highly-oscillatory integrals and their approximations. These approximations will then allow us to construct a uniformly accurate second-order scheme in Section \ref{sec:USDD}. The rigorous convergence analysis is given in Section \ref{sec:convA2}.

\subsubsection{Regularity and expansion of the exact solution}\label{submerge}
In order to derive a second-order scheme, we need to impose additional regularity on the exact solution $\ua(t)$ of \eqref{eq:ua}.
\begin{ass}\label{ass:reg2}
Fix $r>d/2$ and assume that $\ua \in \mathcal{C}([0,T];H^{r+4}(\mathbb{T}^d))$ and in particular
\[
\sup_{0\leq t \leq T} \Vert \ua(t) \Vert_{r+4} \leq M_4 \quad \text{uniformly in $c$}.
\]
\end{ass}
In Lemma  \ref{lem:doubleInt} below we derive two useful expansions of the exact solution $\ua$ of \eqref{eq:ua}. For this purpose we introduce the following definition.
\begin{defn}\label{def:psi} For some function $v$ and $t_n,t \in \mathbb{R}$ we set
\begin{equation}\label{psidef}
\begin{aligned}
  \Psi_{c^2}(t_n,t,v) 
  &:= 
t \e^{2ic^2 t_n} \varphi_1\left(2ic^2t\right) v^3 + 3 t \e^{-2ic^2t_n} \varphi_1\left(-2ic^2 t\right)\vert v \vert^2 \overline{v}+ t \e^{-4ic^2t_n} \varphi_1\left(-4ic^2t\right) \overline{v}^3.
%\\&\frac{1}{2i c^2} \left(\e^{2ic^2(t_n+t)} - \e^{2ic^2 t_n}\right) v^3 + \frac{3}{-2ic^2} \left( \e^{-2ic^2 (t_n+t)} - \e^{-2ic^2 t_n} \right) \vert v\vert^2 \overline{v} \\&+ \frac{1}{-4ic^2}\left(\e^{-4ic^2(t_n+t)}-\e^{-4ic^2 t_n}\right) \overline{v}^3.
\end{aligned}
\end{equation}
\end{defn}
%\begin{rem}
%With the above definition, the first-order scheme \eqref{scheme1} (for $\ua = \va$) may be written  in compact form  as
%\begin{align*}
%\ua^{n+1} = \e^{i \tau \Ac} \Big(\e^{- i \tau \frac{3}{8} \vert \ua^n\vert^2} \ua^n + \tau \frac{3i}{8} 
%\vert \ua^n\vert^2 \ua^n\Big) -  \frac{i}{8} c\nab^{-1} \e^{i \tau \Ac} \Big(\Psi_{c^2}(t_n,\tau,\ua^n) + 3\tau  \vert \ua^n\vert^2 \ua^n
%\Big).
%\end{align*}
%\end{rem}
 The above defintion allows us the following expansions of the exact solution $\ua$.
\begin{lem}\label{lem:doubleInt}
Fix $r>d/2$. Then the exact solution of \eqref{eq:ua} satisfies the expansions
\begin{equation*}
\begin{aligned}
\ua(t_n+s) 
& =  \mathrm{e}^{i s \Ac} \ua(t_n)- \frac{3i}{8}c\nab^{-1} \int_0^s \e^{i(s-\xi)\Ac}
\left \vert \e^{i\xi\Ac}\ua(t_n) \right\vert^2 \left(\e^{i\xi\Ac}\ua(t_n) \right)
  \dd\xi
 \\& - \frac{i}{8}c\nab^{-1}  \Psi_{c^2}(t_n,s,\ua(t_n))+ \mathcal{R}_1(t_n,s, \ua)
\end{aligned}
\end{equation*}
and
\begin{equation*}
\begin{aligned}
\ua(t_n+s) 
& =  \mathrm{e}^{i s \Ac} \ua(t_n)- \frac{i}{8}c\nab^{-1} \Big( 3s 
\left \vert\ua(t_n) \right\vert^2 \ua(t_n) + \Psi_{c^2}(t_n,s,\ua(t_n))\Big)\\&+ \mathcal{R}_2(t_n,s, \ua)
\end{aligned}
\end{equation*}
with $\Psi_{c^2}$ defined in \eqref{psidef} and where the remainders satisfy
\begin{equation}\label{remBdoubleInt}
\begin{aligned}
& \Vert  \mathcal{R}_{1}(t_n,s, \ua)\Vert_r +\Vert  \mathcal{R}_{2}(t_n,s, \ua)\Vert_r \leq s^2 k_{r}(M_2)
\end{aligned}
\end{equation}
for some constant $k_r(M_2)$ which depends on $M_2$, but is independent of $c$.
\end{lem}
\begin{proof}
Note that by Duhamel's perturbation formula \eqref{du} we have that
\begin{equation}
\begin{aligned}
  \ua(t_n+s) =\mathrm{e}^{i s \Ac} \ua(t_n)
%  \\&=- \frac{i}{8}c\nab^{-1}\int_0^s \e^{i(s-\xi)\Ac} \e^{-ic^2 (t_n+\xi)} \left(  \e^{ic^2 (t_n+\xi)}\ua(t_n+\xi) +  \e^{-ic^2 (t_n+\xi)} \overline{\ua}(t_n+\xi))\right)^3\mathrm{d}\xi
- \frac{i}{8}c\nab^{-1} \int_0^s \e^{i(s-\xi)\Ac}\Big(3 \left\vert \ua(t_n+\xi)\right\vert^2 \ua(t_n+\xi)+ \e^{2ic^2(t_n+\xi)} \ua(t_n+\xi)^3 \\
+ 3 \e^{-2ic^2(t_n+\xi)} \left\vert \ua(t_n+\xi)\right\vert^2 \overline{\ua}(t_n+\xi)+ \e^{-4ic^2(t_n+\xi)} \overline{\ua}(t_n+\xi)^3
\Big)\dd\xi.
\end{aligned}
\end{equation}
Therefore, the bound on $c\nab^{-1}$ given in \eqref{cnabm} in particular implies that for $\xi \in \mathbb{R}$ 
\[
\Vert \ua(t_n+\xi) - \e^{i \xi \Ac} \ua(t_n)\Vert_r \leq \xi k_r (1+M_0)^3
\]
for some constant $k_r$ which is independent of $c$. Together with Lemma \ref{lem:expo} and \ref{lem:upc} the assertion then follows by integrating the highly-oscillatory phases $\mathrm{exp}\left( \pm i l c^2 \xi\right)$ exactly.
\end{proof}

 In the next section we collect some important definitions and useful lemmata on highly-oscillatory integrals.
\subsubsection{Preliminary lemmata on highly-oscillatory integrals}\label{sub:PLD}
 The construction of a second-order approximation to $\ua$ based on the iteration of Duhamel's formula \eqref{du} that holds uniformly in all $c > 0$ leads to interactions of the highly-oscillatory phases $\mathrm{e}^{i c^2 t}$. More precisely, we need to handle highly-oscillatory integrals of type
\begin{equation}\label{inti99}
 \int_0^\tau \e^{i s(\delta c^2- \Ac)}  \left( \e^{i s \Ac} v\right)^l \left( \e^{-i s \Ac} \overline{v}\right)^m \dd s , \qquad \delta \in \{-4, -2, 2\}.
\end{equation}
In order to control these integrals we first need to distinguish the non-resonant case $\delta \in \{-4, -2\}$ where
$$
\forall c > 0,\, k \in \mathbb{N}\, : \quad (\delta c^2 - \Ac)_k = \delta c^2 - c\sqrt{c^2+k^2} + c^2  \neq 0
$$
from the resonant case $\delta = 2$ in which  the operator
$
\delta c^2 - \Ac $
 may become singular.
 
In Lemma \ref{lemI1} we outline how to control the non-resonant case $\delta \in \{-4, -2\}$. Lemma \ref{lemI2}  treats the resonant case $\delta = 2$. 

\begin{lem}\label{lemI1}
Fix $r>d/2$. Then we have for $\delta_1 = -2$ and  $\delta_2 = -4$ that for $j = 1,2$ and $l,m \in \mathbb{N}^*$, 
\begin{equation}
\label{merci}
\begin{aligned}
& \int_0^\tau \e^{i s(\delta_j c^2- \Ac)}  \left( \e^{i s \Ac} v\right)^l \left( \e^{-i s \Ac} \overline{v}\right)^m \dd s  
\\&=\tau \varphi_1\left( i\tau(\delta_j c^2-\Ac)\right)v^l\overline{v}^m + i \tau^2 \varphi_2\left( i\tau(\delta_j c^2-\Ac)\right)\left( l v^{l-1} \overline{v}^m \Ac v -m  v^l \overline{v}^{m-1} \Ac \overline{v}
\right) \\&+  \mathcal{R}(t_n,s, v),
\end{aligned}
\end{equation}
where the remainder satisfies
\begin{equation}\label{lemr1}
\Vert   \mathcal{R}(t_n,s, v)\Vert_{r} \leq k_r \tau^3 \Vert v\Vert_{r+4} \Vert v\Vert_r^{l+m-1}
\end{equation}
for some constant $k_r$ which is independent of $c$.
\end{lem}
\begin{proof}
By Taylor series expansion of $\e^{i s \Ac}$ and noting \eqref{boundAc} we obtain that
\begin{equation}\label{inti1}
\begin{aligned}
&\int_0^\tau \e^{-i s \Ac}\e^{i \delta_j c^2 s}   \left( \e^{i s \Ac} v\right)^l \left( \e^{-i s \Ac} \overline{v}\right)^m\dd s  \\
&=  \int_0^\tau \e^{i s ( \delta_j c^2 - \Ac)}
\left(
v^{l} \overline{v}^m + i s \left(l  v^{l-1} \overline{v}^m \Ac v - mv^l \overline{v}^{m-1} \Ac \overline{v}
\right)
\right)
\dd s + \mathcal{R}(t_n,s, v),
\end{aligned}
\end{equation}
where thanks to \eqref{boundAc}  we have for $r>d/2$ that \eqref{lemr1} holds for the remainder. The assertion then follows by the definition of the $\varphi_j$ functions given in Definition \ref{def:phi}.
\end{proof}
 As our numerical scheme will be built on the approximation in \eqref{merci} we need to guarantee that the constructed term
$$
 \tau^2 \varphi_2\left( i\tau(\delta_j c^2-\Ac)\right)\left( l v^{l-1} \overline{v}^m \Ac v -m  v^l \overline{v}^{m-1} \Ac \overline{v}
\right)
$$
is uniformly bounded with respect to $c$ in $H^r$ for all functions $v \in H^r$. This stability analysis is carried out in Remark \ref{remidemi} below, where we in particular exploit the bilinear estimate
\begin{equation}\label{bili}
\textstyle \Vert v w \Vert_r \leq k\, \Vert v \Vert_{r_1} \Vert w \Vert_{r_2} \quad \text{for all } r \leq r_1+r_2-\frac{d}{2} \quad \text{with}\quad
r_1,r_2,-r \neq \frac{d}{2} \quad \text{and}\quad r_1+r_2 \geq 0.
\end{equation}
\begin{rem}[Stability in Lemma \ref{lemI1}]\label{remidemi}
Note that for $\delta_1 = -2$, respectively, $\delta_2 = -4$ we have that
\begin{equation}\label{bbA}
 0 \neq \delta_j c^2 - \Ac  = \delta_j c^2 - c\nab + c^2 = \left\{
\begin{array}{ll}
- (c^2 + c \nab) &\mbox{if} \quad j = 1\\
- (3c^2 + c\nab) & \mbox{if}\quad j = 2
\end{array}
\right. .
\end{equation}
Thanks to \eqref{bbA} which in particular implies that
\begin{align*}
&  \left(\nab\right)_k =  \sqrt{c^2 + \vert k\vert^2} \leq  \sqrt{c^2} + \sqrt{|k|^2} = c + |k| \quad \mbox{and} \quad
\\
& \frac{1}{c^2 + c\left(\nab\right)_k } \leq \mathrm{min}\left\{ |c|^{-2}, |c\sqrt{c^2+k^2}|^{-1}\right\} \leq \mathrm{min}\left\{ |c|^{-2}, (c |k|)^{-1}\right\}
\end{align*}
 we obtain together with the bilinear estimate \eqref{bili}  that for $\delta_j = -2,-4$
\begin{equation}\label{stab21}
\begin{aligned}
& \left\Vert \tau^2 \varphi_2\left(i \tau(\delta_j c^2-\Ac)\right) \left(v \Ac w\right) \right\Vert_r 
 = \tau \left \Vert \frac{\varphi_0(i\tau(\delta_j c^2-\Ac)) - \varphi_1(i\tau(\delta_j c^2-\Ac))}{(\delta_j c^2-\Ac)}  \left(v \Ac w\right)\right\Vert_r
\\&\qquad \leq 2\tau \left
\Vert \frac{1}{(c^2  + c \nab)} \left( v \Ac w \right) \right\Vert_r  \leq 2 \tau 
\left
\Vert \frac{1}{(c^2  + c \nab)} \left( v2 c^2 w \right) \right\Vert_r + 2\tau \left
\Vert \frac{1}{(c^2  + c \nab)} \left( v c \langle \nabla \rangle_0 w \right) \right\Vert_r
\\
&\qquad \leq 4k_r\tau \Vert v \Vert_r \Vert w \Vert_r
\end{aligned}
\end{equation}
for all $r>d/2$ and all functions $v$ and $w$ and some constant $k_r>0$.  The estimate \eqref{stab21} guarantees stability of our numerical  scheme  built  on the approximation in \eqref{merci}.
\end{rem}
 A simple manipulation allows us to treat the resonant case, i.e., $\delta = 2$ in \eqref{inti99}, similarly to Lemma \ref{lemI1}. 
\begin{lem}\label{lemI2}
Fix $r>d/2$ and let $c \neq 0$. Then we have that
\begin{equation}
\begin{aligned}
&\int_0^\tau \e^{i s (2c^2-\Ac)}   \left( \e^{i s \Ac} v\right)^l \left( \e^{-i s \Ac} \overline{v}\right)^m\dd s
=\textstyle \tau \varphi_1\left( i \tau  ( 2c^2 - \frac{1}{2} \Delta) \right) \left(v^l \overline{v}^m\right)
\\&\qquad\textstyle+i  \tau^2 \varphi_2\left( i \tau ( 2c^2 - \frac{1}{2} \Delta)\right)  \Big[
 (\frac12\Delta - \Ac) \left(v^l \overline{v}^m\right) + \left( l v^{l-1} \overline{v}^m \Ac v - m v^l \overline{v}^{m-1} \Ac \overline{v}
\right) 
\Big] \\
& \qquad + \mathcal{R}(t_n,s, v),
\end{aligned}\label{uuli}
\end{equation}
where the remainder satisfies
\begin{equation}\label{lemr2}
\Vert   \mathcal{R}(t_n,s, v)\Vert_{r} \leq k_r \tau^3 \Vert v\Vert_{r+4} \Vert v\Vert_r^{l+m-1}
\end{equation}
for some constant $k_r$ which is independent of $c$.
\end{lem}
\begin{proof}
Note that as
\[
2c^2 - \Ac =\textstyle 2c^2 - \frac{1}{2} \Delta + \frac{1}{2} \Delta-\Ac
\]
we obtain
\begin{equation}
\begin{aligned}\label{int22}
&\int_0^\tau \e^{i s (2c^2-\Ac)}   \left( \e^{i s \Ac} v\right)^l \left( \e^{-i s \Ac} \overline{v}\right)^m\dd s = 
 \int_0^\tau \e^{  i s ( 2c^2 - \frac{1}{2} \Delta)} \e^{i s (\frac{1}{2}\Delta-\Ac)} \left( \e^{i s \Ac} v\right)^l \left( \e^{-i s \Ac} \overline{v}\right)^m\dd s\\
 & =  \int_0^\tau \e^{  i s ( 2c^2 - \frac{1}{2} \Delta)} 
 \Big[ \textstyle\big(1+ i s(\frac12 \Delta-\Ac)\big) \left(v^l\overline{v}^m\right)
+ i s \left( lv^{l-1} \overline{v}^m \Ac v - m v^l \overline{v}^{m-1} \Ac \overline{v}
\right)
 \Big]\dd s+  \mathcal{R}(t_n,s, v),
\end{aligned}
\end{equation}
where thanks to \eqref{boundAc}  we have for $r>d/2$ that \eqref{lemr2} holds for the remainder. The assertion then follows by the definition of the $\varphi_j$ functions given in Definition \ref{def:phi}.
\end{proof}
 Again we need to verify that the constructed term
$$
\textstyle \tau^2 \varphi_2\left( i \tau ( 2c^2 - \frac{1}{2} \Delta)\right)  \Big[
 (\frac12\Delta - \Ac) \left(v^l \overline{v}^m\right) + \left( l v^{l-1} \overline{v}^m \Ac v - m v^l \overline{v}^{m-1} \Ac \overline{v}
\right) 
\Big]
$$
in \eqref{uuli} can be bounded uniformly with respect to $c$ in $H^r$ for all functions $v \in H^r$. This is done in the following remark.
\begin{rem}[Stability in Lemma \ref{lemI2}]
Note that the operator $2c^2-\frac12\Delta $ satisfies the bounds
\[
\frac{ c |k| }{\left(2c^2-\frac12\Delta\right)_k} = \frac{c |k|}{2c^2 + \frac12 \vert k \vert^2} \leq 2, \qquad \frac{ c^2 }{\left(2c^2-\frac12\Delta\right)_k} = \frac{c^2}{2c^2 + \frac12 \vert k \vert^2} \leq \frac12
\]
and furthermore
\[
\left( \Ac \right)_k = c \sqrt{c^2+|k|^2} - c^2 \leq 2c^2+ c |k|.
\]
The above estimates together with the bilinear estimate \eqref{bili} imply that for $r>d/2$ 
\begin{equation}\label{stab22}
\begin{aligned}
& \left \Vert
\textstyle \tau^2 \varphi_2\left( i \tau  ( 2c^2 - \frac{1}{2} \Delta) \right) \left(v \Ac w \right) \right \Vert_r^2 \leq \tau
  \sum_{k} \frac{( 1 + |k|^2)^r}{(2c^2 + \frac12|k|^2)^2} \Big| \sum_{k = k_1 + k_2} v_{k_1} (\Ac)_{k_2} w_{k_2} \Big|^2\\
%& \qquad \leq   \sum_{k} \frac{( 1 + |k|^2)^r}{(2c^2 + \frac12|k|^2)^2} \Big(\sum_{k = k_1 + k_2} |v_{k_1}| (2c^2 + c |k_2|) |w_{k_2}|\Big)^2\\
& \leq \tau  m_r\sum_{k} \frac{( 1 + |k|^2)^r c^4}{(2c^2 + \frac12|k|^2)^2} \Big(\sum_{k = k_1 + k_2} |v_{k_1}| |w_{k_2}|\Big)^2+ \tau  m_r \sum_{k} \frac{( 1 + |k|^2)^rc^2}{(2c^2 + \frac12|k|^2)^2} \Big(\sum_{k = k_1 + k_2} |v_{k_1}| |k_2| |w_{k_2}|\Big)^2
\\
& \leq  \tau  m_r\sum_{k}( 1 + |k|^2)^r  \Big(\sum_{k = k_1 + k_2} |v_{k_1}| |w_{k_2}|\Big)^2+ \tau  m_r \sum_{k} ( 1 + |k|^2)^{r-1} \Big(\sum_{k = k_1 + k_2} |v_{k_1}| |k_2| |w_{k_2}|\Big)^2
\\
& \leq \tau  m_r \Vert v \Vert_r ^2\Vert w \Vert_r^2 + \tau k_r \Vert v \Vert_r^2 \Vert \partial_x w\Vert_{r-1}^2 \leq \tau k
m_r \Vert v \Vert_r^2 \Vert w \Vert_r^2 
\end{aligned}
\end{equation}
for some constant $m_r>0$ which guarantees stability of the numerical method  built on the approximation in Lemma \ref{lemI2}.
\end{rem}
Next we need to analyze  integrals involving the highly-oscillatory function $\Psi_{c^2}$ defined in \eqref{def:psi}. The following lemma yields a uniform approximation.
\begin{lem}\label{intpsiP}
Fix $r>d/2$. Then for any polynomial $p(v)$ in $v$ and $\overline{v}$ we have that
\begin{equation}
\begin{aligned}
& \int_0^\tau \e^{i(\tau-s)\Ac} p\left(\e^{is\Ac}v\right) c\nab^{-1} \Psi_{c^2}(t_n,s,v)\dd s= \tau^2 p(v)c\nab^{-1}\vartheta_{c^2}(t_n,\tau,v) +  \mathcal{R}(t_n,\tau,v)
\end{aligned}
\end{equation}
with
\begin{equation}\label{defpsip}
\begin{aligned}
 \vartheta_{c^2}(t_n,\tau,v) : & = 
\e^{2ic^2 t_n} \frac{\varphi_1\left(2ic^2\tau\right)-1}{2i\tau c^2}  v^3\\& + 3 \e^{-2ic^2t_n}\frac{ \varphi_1\left(-2ic^2 \tau\right)-1}{-2i\tau c^2} \vert v\vert^2 \overline{v}+\e^{-4ic^2t_n}\frac{ \varphi_1\left(-4ic^2\tau\right)-1}{-4i\tau c^2} \overline{v}^3
\end{aligned}
\end{equation}
and where the remainder satisfies
\begin{equation}\label{Rintpsi}
\left \Vert  \mathcal{R}(t_n,\tau,v)\right\Vert_r \leq k_r \tau^3 \left(1+ \Vert v\Vert_{r+2}\right)^5
\end{equation}
for some constant $k_r$ independent of $c$.
\end{lem}
\begin{proof}
Thanks to the approximation \eqref{approx1} and the fact that $\Psi_{c^2}(t_n,s,\ua(t_n))$ is of order one in $s$ uniformly in $c$ we have that
\begin{equation*}
\begin{aligned}
& \int_0^\tau \e^{i(\tau-s)\Ac} p\left(\e^{is\Ac}v\right) c\nab^{-1} \Psi_{c^2}(t_n,s,v)\dd s\\
& =p\left(v\right) c\nab^{-1}   \int_0^\tau\Psi_{c^2}(t_n,s,v)\dd s+ \mathcal{R}(t_n,\tau,v),
\end{aligned}
\end{equation*}
where the remainder satisfies for $r>d/2$ the bound \eqref{Rintpsi}.
\end{proof}
 Finally, we need to handle the interaction of highly-oscillatory phases $\mathrm{e}^{ilc^2 t}$ with the highly-oscillatory function  $\Psi_{c^2}$  defined in \eqref{def:psi}. 
\begin{lem}\label{lem:intPsic}
Let $c\neq 0$. Then, we have for $l \in \mathbb{N}$ that
\begin{equation}\label{def:om}
\begin{aligned}
\Omega_{c^2,l}(t_n,\tau,v)& := \frac{1}{\tau^2}\int_0^\tau \e^{i l c^2s} \Psi_{c^2}(t_n,s,v) \dd s\\&= 
\e^{2ic^2 t_n}  \frac{\varphi_1\left( (l+2)ic^2\tau\right) -\varphi_1\left(lic^2\tau\right)}{2i\tau c^2} v^3 \\&+  3  \e^{-2ic^2t_n} \frac{\varphi_1\left((l-2)ic^2\tau\right) -\varphi_1\left(lic^2\tau\right)}{-2i\tau c^2}\vert v \vert^2 \overline{v}\\
 & + \e^{-4ic^2t_n} \frac{\varphi_1\left((l-4)ic^2\tau\right)-\varphi_1\left(lic^2\tau\right)}{-4i\tau c^2} \overline{v}^3
\end{aligned}
\end{equation}
as well as that
\begin{equation*}
\int_0^\tau \e^{i lc^2s} s \dd s = \tau^2 \varphi_2( i l c^2 \tau).
\end{equation*}
\end{lem}
\begin{proof}
Note that by Definition \ref{def:psi} we have that
\begin{align*}
 \Psi_{c^2}(t_n,s,v) & =  \e^{2 ic^2 t_n} \frac{\e^{(l+2)ic^2s}-\e^{lic^2s}}{2ic^2} v^3 + 3  \e^{-2 ic^2t_n} \frac{\e^{-2 ic^2s}-\e^{lic^2s}}{-2ic^2}\vert v \vert^2 \overline{v}\\
 & + \e^{-4ic^2t_n} \frac{\e^{(l-4)ic^2s}-\e^{lic^2s}}{-4ic^2} \overline{v}^3
\end{align*}
which implies the assertion by Definition \ref{def:phi}  of $\varphi_1$ and $\varphi_2$.
\end{proof}
 With the above lemmata at hand we can commence the construction of the second-order uniformly accurate scheme.
\subsubsection{Uniform second-order discretization of Duhamel's formula}\label{sec:USDD} Our starting point is again Duhamel's perturbation formula (see \eqref{du}) 
\begin{align*}
 \ua(t_n+\tau) &  = \e^{i \tau \Ac} \ua(t_n)\\ & - \frac{i}{8} c \nab^{-1} \int_0^\tau \e^{i(\tau-s) \Ac} \e^{-i c^2 (t_n+s)} \left(
\e^{ic^2(t_n+s)} \ua(t_n+s) + \e^{-ic^2(t_n+s)} \overline{\ua}(t_n+s)\right)^3 \mathrm{d}s\\
\end{align*}
which we split into two parts by separating the linear plus  classical cubic part $|\ua|^2\ua$ from the  terms involving $\ua^3, \overline \ua^3$ and $|\ua|^2\overline \ua$. More precisely, we set
\begin{equation}
\begin{aligned}\label{duha}
 \ua(t_n+\tau) & =  I_{\ast}(\tau,t_n,\ua) - \frac{i}{8} c \nab^{-1} I_{c^{2}}(\tau,t_n,\ua)
 \end{aligned}
 \end{equation}
 with the linear as well as classical cubic part $|\ua|^2\ua$
 \begin{equation}\label{Iast}
 \begin{aligned}
 I_{\ast}(\tau,t_n,\ua):= \e^{i \tau \Ac} \ua(t_n)  - \frac{3i}{8} c \nab^{-1}  \int_0^\tau \e^{i (\tau-s) \Ac}  \vert \ua(t_n+s)\vert^2 \ua(t_n+s) \mathrm{d}s
 \end{aligned}
\end{equation}
 and the terms involving $\ua^3, \overline \ua^3$ and $|\ua|^2\overline \ua$
 \begin{equation}\label{Ic2}
 \begin{aligned}
I_{c^{2}}(\tau,t_n,\ua):= & \int_0^\tau  \e^{i (\tau-s) \Ac} \Big( \e^{2 ic^2 (t_n+s)} \ua^3(t_n+s)  \\&\qquad\qquad + 3 \e^{-2ic^2(t_n+s)} \vert \ua(t_n+s)\vert^2 \overline{\ua}(t_n+s)+ \e^{-4ic^2 (t_n+s)}\overline{\ua}^3(t_n+s)\Big) \mathrm{d}s.
\end{aligned}
\end{equation}
 In order to obtain a second-order uniformly accurate scheme based on the decomposition \eqref{duha} we need to carefully analyze the highly-oscillatory phases in $I_{\ast}(\tau,t_n,\ua)$ and $I_{c^2}(\tau,t_n,\ua)$. We commence with the analysis of $I_{\ast}(\tau,t_n,\ua)$. \\

\emph{1.) First term $I_{\ast}(\tau,t_n,\ua)$:}   By Lemma \ref{lem:doubleInt} we have that
\begin{equation}\label{u18O}
\begin{aligned}
\ua(t_n+s) 
& =  \mathrm{e}^{i s \Ac} \ua(t_n)- \frac{3i}{8}c\nab^{-1} \int_0^s \e^{i(s-\xi)\Ac}
\left \vert \e^{i\xi\Ac}\ua(t_n) \right\vert^2 \left(\e^{i\xi\Ac}\ua(t_n) \right)
  \dd\xi
 \\& - \frac{i}{8}c\nab^{-1}  \Psi_{c^2}(t_n,s,\ua(t_n))+ \mathcal{R}_1(t_n,s, \ua)
\end{aligned}
\end{equation}
with $\Psi_{c^2}$ defined in \eqref{psidef} and where the remainder $\mathcal{R}_1$ is of order $\mathcal{O}(s^2)$ uniformly in $c$. Plugging the approximation \eqref{u18O} into $I_{\ast}(\tau,t_n,\ua)$ defined in \eqref{Iast} yields that 
\begin{equation}
\begin{aligned}\label{Iast1}
 I_{\ast}(\tau,t_n,\ua) & = \e^{i \tau \Ac} \ua(t_n)  - \frac{3i}{8} c \nab^{-1}  \int_0^\tau \e^{i (\tau-s) \Ac}  \vert \ua(t_n+s)\vert^2 \ua(t_n+s) \mathrm{d}s 
\\
 &= \e^{i \tau \Ac} \ua(t_n) - \frac{3i}{8} c \nab^{-1} I_\ast^{1}(\tau,t_n,\ua) \\& +  \frac{3i}{8} c \nab^{-1} \frac{i}{8} \int_0^\tau \e^{i(\tau-s) \Ac}  \Big\{
2  \left \vert \mathrm{e}^{i s \Ac} \ua(t_n)\right\vert^2 c \nab^{-1}  \Psi_{c^2}(t_n,s,\ua(t_n)) \\&
 - \left(\e^{is\Ac} \ua(t_n)\right)^2 c \nab^{-1} \overline{\Psi_{c^2}}(t_n,s,\ua(t_n))
 \Big \}\dd s
 \\&
+\mathcal{R}(\tau,t_n,\ua),
 \end{aligned}
\end{equation}
where we have set
\begin{equation*}\label{iast11}
\begin{aligned}
I_\ast^{1}(\tau,t_n,\ua) &:=  \int_0^\tau \e^{i(\tau-s) \Ac}    \Big\{
 \left \vert \mathrm{e}^{i s \Ac} \ua(t_n)\right\vert^2 \e^{is \Ac} \ua(t_n) \\& \qquad- \frac{3i}{4}  \left \vert \mathrm{e}^{i s \Ac} \ua(t_n)\right\vert^2 c \nab^{-1} \int_0^s \e^{i(s-\xi)\Ac} \vert  \e^{i \xi \Ac}\ua(t_n)\vert^2  \e^{i \xi \Ac}\ua(t_n) \dd \xi\\&\qquad + \frac{3i}{8}\left( \e^{is\Ac} \ua(t_n)\right)^2 c \nab^{-1} \int_0^s \e^{- i(s-\xi)\Ac} \vert  \e^{i \xi \Ac}\ua(t_n)\vert^2 \e^{-i \xi \Ac} \overline{\ua(t_n)} \dd \xi 
 \Big\} \dd s
\end{aligned}
\end{equation*}
and the remainder satisfies
\begin{equation}\label{r3}
\Vert \mathcal{R}(\tau,t_n,\ua)\Vert_r \leq \tau^3 k_{r}(M_4)
\end{equation}
for some constant $k_r(M_4)$ which depends on $M_4$, but is independent of $c$.

 Lemma \ref{intpsiP} allows us to handle the highly-oscillatory integrals involving the function $\Psi_{c^2}$ in \eqref{Iast1}. Thus, in order to obtain a uniform second-order approximation of $ I_{\ast}(\tau,t_n,\ua)$ it remains to derive a suitable second-order approximation to $I_\ast^{1}(\tau,t_n,\ua)$.  

\emph{1.1.) Approximation of  $I_\ast^{1}(\tau,t_n,\ua)$:} The midpoint rule yields the following approximation
\begin{equation}\label{Iastast}
\begin{aligned}
I_\ast^{1}(\tau,t_n,\ua) 
 &  = \tau  \e^{i\frac{\tau}{2}  \Ac} \Big\{
 \left \vert \mathrm{e}^{i \frac{\tau}{2} \Ac} \ua(t_n)\right\vert^2 \e^{i\frac{\tau}{2}  \Ac} \ua(t_n)\\& 
 \qquad - \frac{3i}{4}  \left \vert \mathrm{e}^{i \frac{\tau}{2} \Ac} \ua(t_n)\right\vert^2 c \nab^{-1} \int_0^{\tau/2} \e^{i(\frac{\tau}{2}-\xi)\Ac} \vert  \e^{i \xi \Ac}\ua(t_n)\vert^2  \e^{i \xi \Ac}\ua(t_n) \dd \xi\\&\qquad + \frac{3i}{8}\left( \e^{i\frac{\tau}{2} \Ac} \ua(t_n)\right)^2 c \nab^{-1} \int_0^{\tau/2} \e^{- i(\frac{\tau}{2}-\xi)\Ac} \vert  \e^{i \xi \Ac}\ua(t_n)\vert^2 \e^{-i \xi \Ac} \overline{\ua(t_n)} \dd \xi 
 \Big\} \\&\qquad+\mathcal{R}(\tau,t_n,\ua(t_n)),
\end{aligned}
\end{equation}
where the remainder satisfies thanks to \eqref{boundAc}  and \eqref{cnabm} that
\begin{equation}\label{r4}
\begin{aligned}
\Vert \mathcal{R}(\tau,t_n,\ua(t_n))\Vert_r \leq \tau^3
k_{r}(M_4)
\end{aligned}
\end{equation}
with $k_r$ independent of $c$.

 Next we approximate  the two remaining integrals  in \eqref{Iastast} with the right rectangular rule, i.e.,
\begin{equation}\label{Iast11}
\begin{aligned}
  \int_0^{\tau/2} \e^{ i(\frac{\tau}{2}-\xi)\Ac} \vert  \e^{i \xi \Ac}\ua(t_n)\vert^2  \e^{ i \xi \Ac}\ua(t_n) \dd \xi = \frac{\tau}{2}  \vert  \e^{i\frac{\tau}{2}\Ac}\ua(t_n)\vert^2  \e^{ i \frac{\tau}{2}\Ac}\ua(t_n)+ \mathcal{R}(\tau,t_n,\ua(t_n)),
\end{aligned}
\end{equation}
where the remainder satisfies again thanks to \eqref{boundAc}  that
\begin{equation}\label{r5}
\begin{aligned}
\Vert \mathcal{R}(\tau,t_n,\ua(t_n))\Vert_r \leq \tau^2
k_{r}(M_4)
\end{aligned}
\end{equation}
with $k_r$ independent of $c$.

Plugging \eqref{Iast11} into \eqref{Iastast} yields, with the notation
\begin{align}\label{Usplit}
\U = \e^{i \frac{\tau}{2}\Ac} \ua(t_n)
\end{align}
that
\begin{equation*}
\begin{aligned}
 I_\ast^{1}(\tau,t_n,\ua)  &  =   \e^{i\frac{\tau}{2}  \Ac} \Big\{\tau
 \left \vert \U \right\vert^2 \U\\& 
 \qquad - \frac{\tau^2}{2} \frac{3i}{4}  \left \vert \U\right\vert^2 c \nab^{-1}\vert  \U\vert^2  \U +\frac{\tau^2}{2} \frac{3i}{8}\U^2 c \nab^{-1} \vert  \U\vert^2 \overline{\U} \Big\} \\&
 \qquad+\mathcal{R}(\tau,t_n,\ua(t_n)),
 \end{aligned}
 \end{equation*}
 where thanks to \eqref{r3}, \eqref{r4} and \eqref{r5} the remainder satisfies the bound
$
\Vert \mathcal{R}(\tau,t_n,\ua(t_n))\Vert_r \leq \tau^2
k_{r}(M_4)
$
with $k_r$ independent of $c$.
  
 In order to obtain asymptotic convergence to the classical Strang splitting scheme \eqref{limitscheme2a} associated to the nonlinear Schr\"odinger limit \eqref{NLSlimit}  we add and subtract the term
 $$
  \e^{i\frac{\tau}{2}  \Ac}\frac{\tau^2}{2} \frac{3i}{8} \vert \U\vert^4 \U
 $$
 in the above approximation of  $I_\ast^{1}(\tau,t_n,\ua)$. This yields that
 
 \begin{equation}\label{midpoint}
\begin{aligned}
 & I_\ast^{1}(\tau,t_n,\ua) \\
   &  =   \e^{i\frac{\tau}{2}  \Ac} \Big\{\tau
 \left \vert \U \right\vert^2 \U - \frac{\tau^2}{2} \frac{3i}{8} \vert \U\vert^4 \U
 \\&
 \qquad - \frac{\tau^2}{2} \frac{3i}{4}  \left \vert \U\right\vert^2  \big(c \nab^{-1}-1\big)\vert  \U\vert^2  \U +\frac{\tau^2}{2} \frac{3i}{8}\U^2 \big(c \nab^{-1}-1\big) \vert  \U\vert^2 \overline{\U} \Big\} \\&
 \qquad+\mathcal{R}(\tau,t_n,\ua(t_n)).
\end{aligned}
\end{equation}
 The above decomposition allows us  a second-order approximation of $I_\ast(\tau,t_n,\ua)$ which holds uniformly in all $c$:

\emph{1.2.) Final approximation of  $I_\ast(\tau,t_n,\ua)$:} 
Plugging \eqref{midpoint} into \eqref{Iast1} yields with the aid of Lemma \ref{intpsiP} that
\begin{equation*}\label{IastFin}
\begin{aligned}
& I_\ast(\tau,t_n,\ua) = \e^{i \frac{\tau}{2}\Ac} \Big\{ \U -\frac{3i}{8} \tau
 \left \vert \U \right\vert^2 \U +\left(-\frac{3i}{8}\right)^2 \frac{\tau^2}{2} \vert \U\vert^4 \U\Big\}\\&
- \tau \frac{3i}{8} \Big(c \nab^{-1}-1\Big) \e^{i\frac{\tau}{2}  \Ac}
 \left \vert \U \right\vert^2 \U
+\tau^2 \theta_{c\nab-1}\left(t_n,\tau,\U\right)\\
  &- \tau^2 \frac{3}{32} c\nab^{-1} \left \vert  \ua(t_n)\right\vert^2c\nab^{-1} \vartheta_{c^2}(t_n,\tau,\ua(t_n)) + \tau^2\frac{3}{64}c\nab^{-1} \left( \ua(t_n)\right)^2 c\nab^{-1}\overline{\vartheta_{c^2}}(t_n,\tau,\ua(t_n))\\&
  + \mathcal{R}(\tau,t_n,\ua)
  \end{aligned}
  \end{equation*}
  
  with a remainder $\mathcal{R}$ of order $\mathcal{O}(\tau^3)$ uniformly in $c$.  The Taylor series expansion
  $
\big \vert 1+ x + \frac{x^2}{2} - \e^x \big \vert = \mathcal{O}(x^3)
$
furthermore allows us the following final representation of $I_\ast$:
  \begin{equation}\label{IastFin}
  \begin{aligned}
   I_\ast(\tau,t_n,\ua)   & = \e^{i \frac{\tau}{2}\Ac}\mathrm{exp}\left(-\frac{3i}{8} \tau \vert \U\vert^2\right)\U
\\&
- \tau \frac{3i}{8} \Big(c \nab^{-1}-1\Big) \e^{i\frac{\tau}{2}  \Ac}
 \left \vert \U \right\vert^2 \U
+\tau^2 \theta_{c\nab-1}\left(t_n,\tau,\U\right)\\
  &- \tau^2 \frac{3}{32} c\nab^{-1} \left \vert  \ua(t_n)\right\vert^2c\nab^{-1} \vartheta_{c^2}(t_n,\tau,\ua(t_n)) \\&+ \tau^2\frac{3}{64}c\nab^{-1} \left( \ua(t_n)\right)^2 c\nab^{-1}\overline{\vartheta_{c^2}}(t_n,\tau,\ua(t_n))\\&
  + \mathcal{R}(\tau,t_n,\ua)
\end{aligned}
\end{equation}
with
\begin{equation}\label{def:theta}
\begin{aligned}
& \theta_{c\nab-1}(t_n,\tau,v) : =
- \frac{1}{2} \frac{9}{64}\e^{i\frac{\tau}{2}  \Ac}  \Big(c \nab^{-1}-1\Big) \left \vert v\right\vert^4 v
\\
&
 -\frac{1}{2} \frac{9}{32}  c \nab^{-1}   \e^{i\frac{\tau}{2}  \Ac}\left \vert v\right\vert^2 \Big(c \nab^{-1} - 1 \Big) \vert v\vert^2  v
  +\frac{1}{2} \frac{9}{64}c \nab^{-1} \e^{i\frac{\tau}{2}  \Ac}v^2 \Big(c \nab^{-1} - 1 \Big) \vert  v\vert^2  \overline{v}
 \\
\end{aligned}
\end{equation}
and where $\vartheta_{c^2}$ is defined in \eqref{defpsip} and the remainder $\mathcal{R}(\tau,t_n,\ua)$ satisfies \begin{equation}\label{r6}
\Vert \mathcal{R}(\tau,t_n,\ua(t_n))\Vert_r \leq \tau^3
k_{r}(M_4)
\end{equation}
with $k_r$ independent of $c$.

The approximation of $I_\ast(\tau,t_n,\ua)$ given in \eqref{IastFin} yields the first terms in our numerical scheme. In order to obtain a full approximation to $\ua(t_n+\tau)$ in \eqref{duha} we next derive a second-order approximation to  $I_{c^2}(\tau,t_n,\ua)$.\\

\emph{2.) Second term $I_{c^2}(\tau,t_n,\ua)$:} Applying the second approximation in Lemma \ref{lem:doubleInt} yields together with Lemma \ref{lem:expo}  and by the definition of  $I_{c^2}(\tau,t_n,\ua)$ in \eqref{Ic2} that
\begin{equation*}\label{du2s}
\begin{aligned}
 I_{c^2}(\tau,t_n,\ua)  & = \int_0^\tau \e^{i(\tau-s)\Ac} \Big\{
\e^{2ic^2(t_n+s)} \left(\e^{is\Ac}\ua(t_n)\right)^3 + 3\e^{-2ic^2(t_n+s)} \left\vert\e^{is\Ac}\ua(t_n)\right\vert^2 \e^{-is\Ac}\overline{\ua}(t_n)  \\&+ \e^{-4ic^2(t_n+s)} \left(\e^{-is\Ac}\overline{\ua}(t_n)\right)^3 \Big\}\dd s
\\&+ \int_0^\tau \Big\{-  \frac{3i}{8}  \e^{2ic^2(t_n+s)}\left(\ua(t_n)\right)^2 c\nab^{-1} \Big[ 3 s \vert \ua(t_n)\vert^2 \ua(t_n) + \Psi_{c^2}(t_n,s,\ua(t_n))\Big]\\
& + 3 \e^{-2ic^2(t_n+s)}\Big(- \frac{i}{8}  \left( \overline{\ua}(t_n)\right)^2 c\nab^{-1} \Big[ 3 s \vert \ua(t_n)\vert^2 \ua(t_n) + \Psi_{c^2}(t_n,s,\ua(t_n))\Big]\\
 & + \frac{2 i}{8}  \left\vert\ua(t_n)\right\vert^2 c\nab^{-1} \Big[ 3 s \vert \ua(t_n)\vert^2 \overline{\ua}(t_n) + \overline{\Psi}_{c^2}(t_n,s,\ua(t_n))\Big]
\Big)\\
& +\frac{3i}{8}\e^{-4ic^2(t_n+s)} \left( \overline{\ua}(t_n)\right)^2 c\nab^{-1} \Big[ 3 s \vert \ua(t_n)\vert^2 \overline{\ua}(t_n) + \overline{\Psi}_{c^2}(t_n,s,\ua(t_n))\Big]\Big\} \dd s\\
& + \mathcal{R}(t_n,\tau,\ua)
\end{aligned}
\end{equation*}
with $\Psi_{c^2}$ defined in \eqref{psidef} and where thanks to Lemma \ref{lem:expo}, \ref{lem:doubleInt} and the fact that $\Psi_{c^2}$ is of order one in $s$ uniformly in $c$ the remainder satisfies
$
\Vert \mathcal{R}(\tau,t_n,\ua(t_n))\Vert_r \leq \tau^3
k_{r}(M_4)
$
with $k_r$ independent of $c$.

 Lemma \ref{lemI1}, \ref{lemI2} together with Lemma \ref{lem:intPsic} thus allow us the following expansion of $I_{c^2}$: We have
\begin{equation}\label{Ic2calc}
\begin{aligned}
& I_{c^2}(\tau,t_n,\ua)  = I^1_{c^2}(\tau,t_n,\ua) + \mathcal{R}(t_n,\tau,\ua)
\end{aligned}
\end{equation}
with the highly-oscillatory term 
\begin{equation}\label{IOkti}
\begin{aligned}
&  I^1_{c^2}(\tau,t_n,\ua): =\textstyle \tau \e^{2 i c^2 t_n} \e^{i \tau \Ac}  \varphi_1\left(i \tau ( 2c^2 - \frac{1}{2} \Delta)\right) \ua^3(t_n)\\& \textstyle+ i \tau^2 \e^{2 i c^2 t_n} e^{i \tau \Ac} \varphi_2\left(i \tau ( 2c^2 - \frac{1}{2} \Delta)\right) \Big[
(\frac12\Delta-\Ac)\ua^3(t_n) + 3 \ua^2(t_n) \Ac \ua(t_n)
\Big]\\
&+ 3\tau \e^{-2ic^2t_n} \e^{i \tau \Ac} \varphi_1(i\tau(-2c^2-\Ac))\left\vert\ua(t_n)\right\vert^2 \overline{\ua}(t_n) \\&+ 3 i \tau^2\e^{-2ic^2t_n}   \e^{i \tau \Ac}\varphi_2(i\tau(-2c^2-\Ac)) \Big[ \overline{\ua}^2(t_n)\Ac \ua(t_n) - 2 \vert \ua(t_n)\vert^2 \Ac \overline{\ua}(t_n)
\Big]\\
& + \tau  \e^{-4ic^2t_n} \e^{i \tau \Ac}\varphi_1(i\tau(-4c^2 - \Ac)) \overline{\ua}^3(t_n) - i \tau^2\e^{-4ic^2t_n} \e^{i \tau \Ac}\varphi_2(i\tau(-4c^2-\Ac)) 3 \overline{\ua}^2(t_n) \Ac \overline{\ua}(t_n)\\
& - \tau^2 \frac{3i}{8}  \e^{2ic^2t_n} \left(\ua(t_n)\right)^2 c\nab^{-1}  \Big[3\varphi_2(2ic^2\tau)\vert \ua(t_n)\vert^2 \ua(t_n) + \Omega_{c^2,2,}(t_n,\tau,\ua(t_n)) \Big]\\
& - \tau^2 \frac{3i}{8}  \e^{-2ic^2t_n}  \left( \overline{\ua}(t_n)\right)^2 c\nab^{-1} \Big[3 \varphi_2(-2ic^2\tau) \vert \ua(t_n)\vert^2 \ua(t_n) + \Omega_{c^2,-2}(t_n,\tau,\ua(t_n))\Big] \\
 & + \tau^2  \frac{6 i}{8}  \e^{-2ic^2t_n}  \left\vert \ua(t_n)\right\vert^2 c\nab^{-1}  \Big[3 \varphi_2(-2ic^2\tau) \vert \ua(t_n)\vert^2 \overline{\ua}(t_n)+ \overline{\Omega}_{c^2,-2}(t_n,\tau,\ua(t_n))
 \Big] \\
& +\tau^2 \frac{ 3 i}{8}\e^{-4ic^2t_n} \left( \overline{\ua}(t_n)\right)^2 c\nab^{-1}  \Big[3 \varphi_2(-4ic^2\tau)\vert \ua(t_n)\vert^2 \overline{\ua}(t_n) + \overline{\Omega}_{c^2,-4}(t_n,\tau,\ua(t_n))\Big]
\\
& + \mathcal{R}(t_n,\tau,\ua),\\
\end{aligned}
\end{equation}
where $\Omega_{c^2,l}$ is defined in Lemma \ref{lem:intPsic} and the remainder satisfies
\begin{equation}\label{R9}
\Vert  \mathcal{R}(t_n,\tau,\ua)\Vert_r \leq \tau^3 k_r(M_4),
\end{equation}
with $k_r$ independent of $c$.\\

\emph{3.) Final approximation of $\ua(t_n+\tau)$:}
Plugging \eqref{IastFin} as well as \eqref{Ic2calc} into \eqref{duha} builds the basis of our second-order scheme: As a numerical approximation  to the exact solution $\ua$ at time $t_{n+1}$ we take the second-order uniform accurate exponential-type integrator: $\Un = \e^{i \frac{\tau}{2}\Ac} \ua^n$ and 
\begin{equation} \label{scheme2}
\begin{aligned}
 \ua^{n+1} & =  \mathrm{e}^{i \frac{\tau}{2} \Ac}
\mathrm{e}^{- i \tau \frac{3}{8}  \left \vert \Un \right\vert^2}
\Un \\&
 - \tau \frac{3i}{8} \Big(c \nab^{-1}-1\Big) \e^{i\frac{\tau}{2}  \Ac}
\left \vert \Un\right\vert^2 \Un
+\tau^2 \theta_{c\nab-1}\left(t_n,\tau,\Un \right)\\&- \tau^2 \frac{3}{64} c\nab^{-1} \Big[2 \left \vert  \ua^n\right\vert^2c\nab^{-1} \vartheta_{c^2}(t_n,\tau,\ua^n) - \left( \ua^n\right)^2 c\nab^{-1}\overline{\vartheta_{c^2}}(t_n,\tau,\ua^n)\Big]\\
& - \frac{i}{8} c \nab^{-1} I^1_{c^2}(\tau,t_n,\ua^n),
\end{aligned}
\end{equation}
where $ I^1_{c^2}(\tau,t_n,\ua^n)$ is defined in \eqref{IOkti}  and with $\varphi_1, \varphi_2$ given in Definition \ref{def:phi}, $\theta_{c\nab-1}$ given in \eqref{def:theta}, $\vartheta_{c^2}$ in \eqref{defpsip} and $\Omega_{c^2,l}$ in \eqref{def:om}.

\subsection{Convergence analysis}\label{sec:convA2} The exponential-type integration scheme \eqref{scheme2} converges (by construction) with second-order in time uniformly with respect to $c$.

\begin{thm}[Convergence bound for the second-order scheme]\label{them:con2}
Fix $r>d/2$ and assume that
\begin{equation}\label{eq:urged2}
\Vert z(0)\Vert_{r+4} + \Vert c^{-1}\nab^{-1}z'(0)\Vert_{r+4} \leq M_4
\end{equation}
uniformly in $c$.
For $\ua^n$ defined in \eqref{scheme2} we set
\[
z^n := \frac{1}{2} \left( \e^{ic^2t_n} \ua^n + \e^{-ic^2 t_n} \overline{\ua^n}\right).
\]
Then, there exists a $T_r>0$ and $\tau_0>0$ such that for all $\tau \leq \tau_0$ and $t_n\leq T_r$ we have for all $c >0$ that
\[
\Vert z(t_n)-z^n\Vert_r \leq \tau^2 K_{1,r,M_4} \e^{t_n K_{2,r,M}} \leq \tau^2 K^\ast_{r,M,M_4,t_n},
\]
where the constants $K_{1,r,M_2},K_{2,r,M}$ and $K^\ast_{r,M,M_4,t_n}$  can be chosen independently of $c$.
\end{thm}  
\begin{proof}
First note that the regularity assumption on the initial data in \eqref{eq:urged2} implies the regularity Assumption \ref{ass:reg2} on $\ua(t)$, i.e., there exists a $T_r>0$ such that
\[
\sup_{0 \leq t \leq T} \Vert \ua(t)\Vert_{r+4} \leq k(M_4)
\]
for some constant $k$ that depends on $M_4$ and $T_r$, but can be chosen independently of $c$.

In the following let $\phi^t$ denote the exact flow of \eqref{eq:ua}, i.e., $\ua(t_{n+1}) = \phi^\tau(\ua(t_n))$ and let $\Phi^\tau$ denote the numerical flow defined in \eqref{scheme2}, i.e.,
\[
\ua^{n+1} = \Phi^\tau(\ua^n).
\]
Taking the difference of \eqref{du} and \eqref{scheme2} yields that
\begin{equation}\label{glob02}
\begin{aligned}
\ua(t_{n+1}) - \ua^{n+1}& = \phi^\tau(\ua(t_n)) - \Phi^\tau(\ua^n)\\
&= \Phi^\tau(\ua(t_n)) - \Phi^\tau(\ua^n) + 
 \phi^\tau(\ua(t_n))
- \Phi^\tau(\ua(t_n)).
\end{aligned}
\end{equation}

\emph{Local error bound:} With the aid of the expansion  \eqref{IastFin} and \eqref{Ic2calc} we obtain by the representation of the exact solution in \eqref{duha} together with the error bounds \eqref{R9}  and \eqref{r6} that
\begin{equation}\label{local2}
\Vert  \phi^\tau(\ua(t_n))
- \Phi^\tau(\ua(t_n)) \Vert_r = \Vert \mathcal{R}(\tau,t_n,\ua)\Vert_r \leq \tau^3 k_r(M_4)
\end{equation}
for some constant $k_r$ which depends on $M_4$ and $r$, but can be chosen independently of $c$.

\emph{Stability bound:} Note that by the definition of $ \varphi_2$ in Definition \ref{def:phi}, $\theta_{c\nab-1}$ in \eqref{def:theta}, $\vartheta_{c^2}$ in \eqref{defpsip} and $\Omega_{c^2,l}$ in \eqref{def:om} we have for $l=-4,-2,2$ that
\begin{equation}\label{opB}
\begin{aligned}
\tau^2 \Big( \Vert \varphi_2(lic^2 t) (f-g)\Vert_r + \Vert \Omega_{c^2,l}(t_n,\tau,f)- \Omega_{c^2,l}(t_n,\tau,g)\Vert_r + \Vert \vartheta_{c^2}(t_n,\tau,f)-\vartheta_{c^2}(t_n,\tau,g) \Vert_r \Big)\\ \leq \tau k_r\left(\Vert f\Vert_r,\Vert g\Vert_r\right) \Vert f - g\Vert_r
\end{aligned}
\end{equation}
for some constant $k_r$ independent of $c$. Together with the bound \eqref{cnabm}, the definition of $\varphi_1$ in Definition \ref{def:phi} and the stability estimates \eqref{stab21}  and \eqref{stab22} we thus obtain as long as $\Vert \ua(t_n)\Vert_r \leq M$ and $\Vert \ua^n\Vert_R \leq 2M$ that
\begin{equation}\label{stab2}
\Vert \Phi^\tau(\ua(t_n)) - \Phi^\tau(\ua^n) \Vert_r \leq \Vert \ua(t_n) - \ua^n\Vert_r +\tau K_{r,M} \Vert \ua(t_n) - \ua^n\Vert_r,
\end{equation}
where the constant $K_{r,M}$ depends on $r$ and $M$, but can be chosen independently of $c$.

\emph{Global error bound:} Plugging the stability bound \eqref{stab2}  as well as the local error bound \eqref{local2} into \eqref{glob02} yields by a bootstrap argument that
\begin{align}\label{conus2}
\left \Vert \ua(t_{n}) - \ua^{n} \right\Vert_r \leq \tau^2 K_{1,r,M_4} \mathrm{e}^{t_n K_{2,r,M}},
\end{align}
where the constants are uniform in $c$. Note that as $u = v$ we have by \eqref{eq:zuv} and \eqref{psi} that
\begin{align*}
\Vert z(t_n) - z^n\Vert_r & = \textstyle \left \Vert\frac12 \big( u(t_n) + \overline{u(t_n)}\big) -\frac12 \big(\e^{ic^2 t_n} \ua^n + \e^{-ic^2t_n} \overline{\ua^n}\big)\right\Vert\\
&  \leq \Vert \e^{ic^2t_n} (\ua(t_n)-\ua^n)\Vert_r   =  \Vert \ua(t_n)-\ua^n\Vert_r  .
\end{align*}
Together with the bound in \eqref{conus2} this completes the proof.
\end{proof}
\begin{rem}[Fractional convergence and convergence in $L^2$]\label{remFrac}
A fractional convergence result as Theorem \ref{them:con1Frac} for the first-order scheme also holds for the second-order exponential-type integrator \eqref{scheme2}: Fix $r>d/2$ and let $0\leq \gamma \leq 1$. Assume that
\[
\Vert z(0)\Vert_{r+2+2\gamma} + \Vert c^{-1}\nab^{-1}z'(0)\Vert_{r+2+2\gamma} \leq M_{2+2\gamma}.
\]
Then, the scheme \eqref{scheme2} is convergent of order $\tau^{1+\gamma}$ in $H^r$ uniformly with respect to $c$.

Furthermore, for initial values satisfying
\[
\Vert z(0)\Vert_{4} + \Vert c^{-1}\nab^{-1}z'(0)\Vert_{4} \leq M_{4,0}
\]
the  exponential-type integration scheme \eqref{scheme2}  is second-order convergent in $L^2$ uniformly with respect to $c$ by the strategy presented in \cite{Lubich08}.
\end{rem}

In analogy to Remark \ref{rem:limit1} we make the following observation: For sufficiently smooth solutions the exponential-type integration scheme \eqref{scheme2} converges in the limit $c \to \infty$ to the classical Strang splitting of  the corresponding nonlinear Schr\"odinger limit equation \eqref{NLSlimit}.

\begin{rem}[Approximation in the non relativistic limit $c \to \infty$]
The exponential-type integration scheme \eqref{scheme2} corresponds for sufficiently smooth solutions in the limit $\ua^n\stackrel{c\to \infty}{\longrightarrow} u_{\ast,\infty}^n$, essentially to the Strang Splitting (\cite{Lubich08,Faou12})
\begin{equation}\label{limitStrang}
\begin{aligned}
u_{\ast, \infty}^{n+1}  &=\e^{-i \frac{\tau}{2} \frac{\Delta}{2}}  \mathrm{e}^{-i \tau \frac{3}{8} \vert \e^{-i \frac{\tau}{2} \frac{\Delta}{2}}u_{\ast, \infty}^n\vert^2}\e^{-i \frac{\tau}{2} \frac{\Delta}{2}}u_{\ast, \infty}^n,\qquad u_{\ast,\infty}^0 = \varphi - i \gamma,
\end{aligned}
\end{equation}
for the cubic nonlinear Schr\"odinger limit system \eqref{NLSlimit}.

More precisely, the following Lemma holds.
\end{rem}

\begin{lem}\label{rem:limit2} 
Fix $r>d/2$. Assume that
\begin{equation*}
\Vert z(0) \Vert_{r+3} +\Vert c^{-1}\nab^{-1} z'(0)\Vert_{r+3} \leq M_{3}
\end{equation*}
for some $\varepsilon>0$ uniformly in $c$ and let the initial value approximation (there exist functions $\varphi,\gamma$ such that)

\begin{align*}
\Vert z(0)- \gamma\Vert_r + \Vert c^{-1}\nab^{-1} z'(0) - \varphi\Vert_r \leq k_r c^{-1}
\end{align*}
hold for some constant $k_r$ independent of $c$.

Then, there exists a $T>0$ and $\tau_0>0$ such that for all $\tau \leq \tau_0$ the difference of the second-order scheme \eqref{scheme2} for system \eqref{eq:ua} and the Strang splitting \eqref{limitStrang} for the limit Schr\"odinger equation \eqref{NLSlimit} satisfies for $ t_n \leq T$ and all $c >0$ with
\begin{equation*}
\tau c \geq 1
\end{equation*}
that
\[
\Vert \ua^n- u_{\ast, \infty}^{n} \Vert_r  \leq c^{-1} k_r(M_{3},T)
\]
for some constant $k_{r}$ that depends on $M_{3}$ and $T$, but is independent of $c$.
\end{lem}
\begin{proof}
The proof follows the line of argumentation to the proof of Lemma \ref{rem:limit1} by noting that for $l=-4,-2$ and $n=-4,-2,2$
\[
\tau  \Big( \Vert \varphi_j(2i \tau \nab^2) \Vert_r + \Vert \varphi_j\big( i \tau (lc^2 - \Ac)\big)\Vert_r + \Vert \varphi_j( n i c^2 \tau) \Vert_r \Big) \leq k_r c^{-2}
\]
for some constant $k_r$ independent of $c$.
\end{proof}

\subsection{Simplifications in the ``weakly to strongly  non-relativistic limit regime''}\label{sec:limit2}
In the ``weakly to strongly non-relativistic limit regime'', i.e., for large values of $c$, we may again (substantially) simplify the second-order scheme \eqref{scheme2} and nevertheless obtain a well suited, second-order approximation to $\ua(t_n)$ in \eqref{eq:ua}.

\begin{rem}[Limit scheme \cite{FS13}]\label{remarkStrangOk19}
For sufficiently large values of $c$ and sufficiently smooth solutions, more precisely, if \[
 \Vert z(0) \Vert_{r+4} + \Vert c^{-1}\nab^{-1} z'(0)\Vert_{r+4}\leq M_{4}\quad \text{and}\quad \tau c > 1
\]
we may take  instead of \eqref{scheme2}  the classical Strang splitting (see  \cite{Lubich08,Faou12}) for the nonlinear Schr\"odinger limit equation \eqref{NLSlimit}, namely,
\begin{equation}\label{limitscheme2}
\begin{aligned}
u_{\ast,\infty}^{n+1}  &=\e^{-i \frac{\tau}{2} \frac{\Delta}{2} }  \mathrm{e}^{-i \tau \frac{3}{8}\vert \e^{-i \frac{\tau}{2} \frac{\Delta}{2} }  u_{\ast,\infty}^n\vert^2  }\e^{-i \frac{\tau}{2} \frac{\Delta}{2} }  u_{\ast,\infty}^n\\
\end{aligned}
\end{equation}
as a second-order numerical approximation to $\ua(t_n)$ in \eqref{eq:ua}. The assertion follows from \cite{FS13} thanks to the approximation
\begin{align*}
\Vert \ua(t_n) -  u_{\ast,\infty}^{n} \Vert_r & \leq \Vert \ua(t_n) - u_{\ast,\infty}(t_n)\Vert_r  + \Vert u_{\ast,\infty}(t_n) - u_{\ast,\infty}^{n} \Vert_r = \mathcal{O}\big( c^{-2}+\tau^2\big).
\end{align*}
\end{rem}

\section{Numerical experiments}
In this section we numerically confirm first-, respectively, second-order convergence uniformly in $c$ of the exponential-type integration  schemes \eqref{scheme1}  and \eqref{scheme2}. In the numerical experiments we use a standard Fourier pseudospectral method for the space discretization with the largest Fourier mode $K = 2^{10}$ (i.e., the spatial mesh size $\Delta x =  0.0061$) and integrate up to $T = 0.1$. In Figure \ref{fig1} we plot (double logarithmic) the time-step size versus the error measured in a discrete $H^1$ norm of the first-order scheme  \eqref{scheme1} and the second-order scheme \eqref{scheme2} with initial values
\begin{align*}
 & z(0,x) = \frac12 \frac{ \mathrm{cos}(3 x)^2 \mathrm{sin}(2x)}{2-\mathrm{cos}(x)},\qquad \partial_t z(0,x) = c^2 \frac12 \frac{\mathrm{sin}(x)\mathrm{cos}(2x)}{2-\mathrm{cos}(x)}
\end{align*}
for different values $c = 1, 5, 10,  50, 100, 500, 1000, 5000, 10000 .$

\begin{figure}[h!]
\centering
\includegraphics[width=0.5\linewidth]{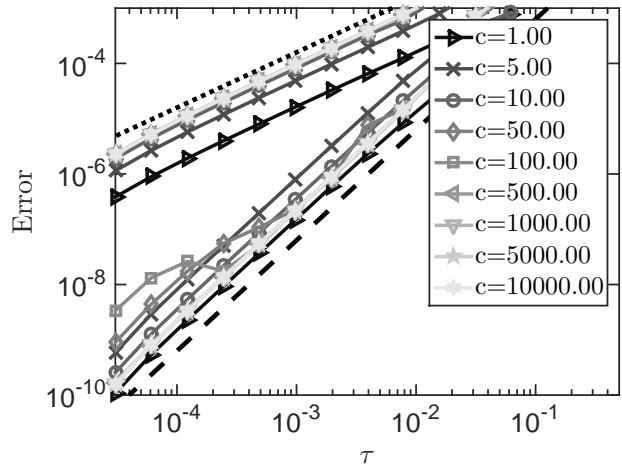}
\caption{Error of the first-, respectively, second-order exponential-type integration scheme \eqref{scheme1} and \eqref{scheme2}. The slope of the dotted and dashed line is one and two, respectively.}\label{fig1}
\end{figure}

\section*{Acknowledgement}
The authors gratefully acknowledge financial support by the Deutsche Forschungsgemeinschaft (DFG) through CRC 1173. This work was also partly supported by the ERC Starting Grant Project GEOPARDI No 279389.

\end{document}